\documentclass[11pt]{amsart}
\usepackage[dvips]{graphicx,color}
\usepackage{amssymb,amscd,latexsym,epsfig,hyperref,xcolor}
\usepackage[curve,color]{xypic}
\usepackage[mathscr]{euscript}
\usepackage{stmaryrd}

\DeclareFontFamily{OT1}{rsfs}{}
\DeclareFontShape{OT1}{rsfs}{n}{it}{<-> rsfs10}{}
\DeclareMathAlphabet{\curly}{OT1}{rsfs}{n}{it}

\makeatletter
\newcommand{\eqnum}{\refstepcounter{equation}\textup{\tagform@{\theequation}}}
\makeatother

\renewcommand\;{\hspace{.6pt}}

\newcommand\PP{\mathbb P}
\newcommand\LL{\mathbb L}
\newcommand\C{\mathbb C}
\newcommand\Q{\mathbb Q}

\newcommand\Z{\mathbb Z}

\newcommand\cO{\mathcal O}

\renewcommand\cD{\mathcal D}

\newcommand\cF{\mathcal F}

\newcommand\cI{\mathcal I}

\renewcommand\cL{\mathcal L}

\newcommand\cQ{\mathcal Q}

\newcommand\cU{\hspace{1pt}\mathcal U}

\renewcommand\t{\mathfrak t}

\makeatletter
\newcommand{\so}{\ \ext@arrow 0359\Rightarrowfill@{}{\hspace{3mm}}\ }
\makeatother
\newcommand{\rt}[1]{\xrightarrow{\ #1\ }}
\newcommand\To{\longrightarrow}

\newcommand\into{\hookrightarrow}
\newcommand\INTO{\ \ar@{^(->}[r]<-.2ex>}
\newcommand{\Into}{\,\ensuremath{\lhook\joinrel\relbar\joinrel\rightarrow}\,}
\renewcommand\Mapsto{\ensuremath{\shortmid\joinrel\relbar\joinrel\rightarrow}}

\renewcommand\_{^{}_}
\newcommand\take{\backslash}

\newfont{\bigtimesfont}{cmsy10 scaled \magstep5}
\newcommand{\bigtimes}{\mathop{\lower0.9ex\hbox{\bigtimesfont\symbol2}}}
\renewcommand\={\ =\ }

\newcommand{\wwedge}{\mbox{\Large $\wedge$}}

\newcommand\At{\operatorname{At}}

\newcommand\Gr{\operatorname{Gr}}

\newcommand\rk{\operatorname{rank}}
\newcommand\vir{\operatorname{vir}}
\newcommand\vd{\operatorname{vd}}
\newcommand\tr{\operatorname{tr}}
\newcommand\coker{\operatorname{coker}}
\newcommand\im{\operatorname{im}}
\newcommand\id{\operatorname{id}}

\newcommand\Hom{\operatorname{Hom}}
\renewcommand\hom{\curly H\!om}

\newcommand\Ext{\operatorname{Ext}}
\newcommand\ext{\curly E\!\;xt}

\newcommand\Pic{\operatorname{Pic}}

\newcommand\Proj{\operatorname{Proj}\,}

\newcommand\Sym{\operatorname{Sym}}
\newcommand\Cone{\operatorname{Cone}}

\newcommand\beq[1]{\begin{equation}\label{#1}}
\newcommand\eeq{\end{equation}}
\newcommand\beqa{\begin{eqnarray*}}
\newcommand\eeqa{\end{eqnarray*}}

\newcommand\arXiv[1]{\href{http://arxiv.org/abs/#1}{arXiv:#1}}
\newcommand\mathAG[1]{\href{http://arxiv.org/abs/math/#1}{math.AG/#1}}

\renewcommand{\S}[1]{S^{[#1]}}
\renewcommand{\a}{\alpha}
\renewcommand{\b}{\beta}
\newcommand{\s}{\sigma}

\newcommand{\bb}[1]{\mathbb{#1}}

\newcommand{\cc}[1]{\mathcal{#1}}

\newcommand{\xr}[1]{\xrightarrow{#1}}
\newcommand{\wt}[1]{\widetilde{#1}}

\newcommand{\op}[1]{\operatorname{#1}}

\newcommand{\onto}[1]{\xymatrix{\ar@{->>}[r]^-{#1} & }}

\newcommand{\bu}{\bullet}

\DeclareRobustCommand{\SkipTocEntry}[3]{}
\makeatletter
\newcommand\@dotsep{4.5}
\def\@tocline#1#2#3#4#5#6#7{\relax
  \ifnum #1>\c@tocdepth 
  \else
    \par \addpenalty\@secpenalty\addvspace{#2}%
    \begingroup \hyphenpenalty\@M
    \@ifempty{#4}{%
      \@tempdima\csname r@tocindent\number#1\endcsname\relax
    }{%
      \@tempdima#4\relax
    }%
    \parindent\z@ \leftskip#3\relax \advance\leftskip\@tempdima\relax
    \rightskip\@pnumwidth plus1em \parfillskip-\@pnumwidth
    #5\leavevmode #6\relax
    \leaders\hbox{$\m@th
      \mkern \@dotsep mu\hbox{.}\mkern \@dotsep mu$}\hfill
    \hbox to\@pnumwidth{\@tocpagenum{#7}}\par
    \nobreak
    \endgroup
  \fi}
\makeatother

\makeatletter \@addtoreset{equation}{section} \makeatother
\renewcommand{\theequation}{\thesection.\arabic{equation}}

\newtheorem{thm}[equation]{Theorem}
\newtheorem{thm*}{Theorem}
\newtheorem{lem}[equation]{Lemma}
\newtheorem{cor}[equation]{Corollary}
\newtheorem{prop}[equation]{Proposition}
\newtheorem{prop*}{Proposition}

\newenvironment{rmk}{\noindent\textbf{Remark}.}{\medskip}
\newenvironment{rmks}{\noindent\textbf{Remarks}.}{\medskip}

\title{Degeneracy loci, virtual cycles and nested Hilbert schemes II}
\author{Amin Gholampour and Richard P. Thomas}

\begin{document}
\maketitle
\begin{abstract} We express nested Hilbert schemes of points and curves on a smooth projective surface as ``virtual resolutions" of degeneracy loci of maps of vector bundles on smooth ambient spaces. 

We show how to modify the resulting obstruction theories to produce the virtual cycles of Vafa-Witten theory and other sheaf-counting problems. The result is an effective way of calculating invariants (VW, SW, local PT and local DT) via Thom-Porteous-like Chern class formulae.
\end{abstract}

\renewcommand\contentsname{\vspace{-9mm}}
\tableofcontents \vspace{-1cm}


\section{Introduction and summary of results} \label{intro}
\noindent\textbf{Overview.}
It is usually hard to calculate with the virtual fundamental class $[M]^{\vir}$ of a moduli space $M$. In the rare situation that $M$ (and its perfect obstruction theory) is cut out of a smooth ambient space $\iota\colon M\into X$ by a section of a vector bundle $E$, the push forward of the virtual cycle $\iota_*[M]^{\vir}=e(E)$ is the Euler class of $E$, making calculation possible.

In \cite{GT1} we described a generalisation to the case where $M$ is instead the \emph{deepest} degeneracy locus of a map of vector bundles on a smooth ambient space. This we applied to nested Hilbert schemes of \emph{points} on smooth surfaces, reproducing the virtual cycles of \cite{GSY1} --- which agree with those arising in Vafa-Witten theory \cite{TT1} or reduced local DT theory \cite{GSY2}.

That technique does not extend to other degeneracy loci, but it does apply to their \emph{resolutions} \eqref{resl}. We apply this to nested Hilbert schemes of \emph{points and curves} on smooth surfaces with their virtual cycles \cite{GSY1} coming from Vafa-Witten theory \cite{TT1} or reduced DT theory \cite{GSY2}. The result is an effective new tool for computing sheaf-theoretic enumerative invariants of projective surfaces, such as those arising in Seiberg-Witten, Vafa-Witten, local DT and local PT theories.

\medskip\noindent\textbf{Degeneracy loci.} Let $X$ be a smooth complex quasi-projective variety and
\beq{2term}
\s\,\colon\ E_0\To E_1
\eeq
a map of vector bundles of ranks $e_0$ and $e_1$ over $X$. For a positive integer $r\le e_0$, consider the $r$\;th degeneracy locus of $\s $,
$$
D_r(\s)\ :=\ \Big\{ x\in X \colon\dim\ker\;(\s_x) \ge r \Big \}
$$
with scheme structure given by viewing it as the zero locus of
$$
\wwedge^{e_0-r+1}\s\,\colon\ \wwedge^{e_0-r+1} E_0\To \wwedge^{e_0-r+1} E_1.
$$
That is, $D_r(\sigma)$ is the $(e_1-e_0+r-1)$th Fitting scheme of the first cohomology sheaf $h^1(E_\bu)=\coker\sigma$ of the complex $E_\bu$ \eqref{2term}. In particular it depends on $E_\bu$ only up to quasi-isomorphism, and we often denote it $D_r(E_\bu)$.

For the nested Hilbert scheme applications in this paper we will take only $r=1$. On a first read-through of the abstract construction the reader may find it simpler to do the same, substituting projective bundles $\PP(-)$ for the Grassmann bundles $\Gr(r,-)$ that follow.

\medskip\noindent\textbf{Resolutions.} Over points $x\in D_r$ with $\dim\ker\;(\s_x)=r$ there is a \emph{unique} $r$-dimensional subspace of $E_0|_x$ on which $\sigma_x$ vanishes. Taking the set of all such $r$-dimensional subspaces for \emph{every} point of $D_r$ defines a natural space $\wt D_r$ dominating $D_r$,\beq{resl}
\xymatrix@=15pt{
\wt D_r(\s) \ar[d]_p& \hspace{-6mm}:=\ \big\{(x,U)\in\Gr\;(r,E_0)\ \colon\ U \subseteq \ker\;(\s_x)\subseteq E_0|_x\big\} \\
D_r(\s),}\vspace{-1mm}
\eeq
whose fibre over $x$ is $\Gr\;(r,\ker\s_x)$. Its scheme structure\footnote{A more efficient description of $\wt D_r(\sigma)$ is as the relative Grassmannian $\Gr\;(\coker\sigma^*,r)$ of $r$-dimensional quotients of the sheaf $\coker(\sigma^*)$.} comes from viewing  $\wt D_r(\s)\subset\Gr\;(r,E_0)\rt{p}X$ as the vanishing locus of the composition 
\beq{vanloc}
\cU\Into p^*E_0 \xr{p^* \s} p^*E_1.
\eeq
Here $\cU \subset p^*E_0$ denotes the universal rank $r$ subbundle on $\Gr\;(r,E_0)$ (or its restriction to $\wt D_r$) whose fibre at $(x,U)$ \eqref{resl} is $U$. If $\sigma$ is appropriately transverse then each $D_r$ has codimension $r(e_1-e_0+r)$, singular locus $D_{r+1}$, and $\wt D_r\to D_r$ is a \emph{resolution of singularities}. For arbitrary $\sigma$ we call $\wt D_r$ the \emph{virtual resolution} of $D_r$ since it is \emph{virtually smooth}:

\begin{prop*}\label{mainp} Expressing $\wt D_r(E_\bu)$ as the zero locus  of the section \eqref{vanloc} of the bundle $\cU^*\otimes p^*E_1$ endows it with a perfect obstruction theory
$$
\Big\{p^*T_X\To\cU^*\otimes\mathrm{Cone}\,(\cU\to p^*E_\bullet)[1]\Big\}^\vee\To\LL_{\wt D_r}
$$
of virtual dimension $\vd:=\dim X-r(e_1-e_0+r)$. It depends only on the quasi-isomorphism class of the complex $E_\bullet$ \eqref{2term}. The resulting virtual cycle
$$
\big[\wt D_r(E_\bu)\big]^{\vir}\ \in\ A_{\vd}\big(\wt D_r(E_\bu)\big),
$$
when pushed forward to $X$, has class given by the Thom-Porteous formula\footnote{Here $\Delta^a_b(c):=\det(c\_{b+j-i})_{1\le i,j\le a}\,,$ as in \cite[Chapter 14]{Fu}.}
$$
\Delta\;_{r-\op{rk}(E_\bu)}^r\big(c(-E_\bullet)\big)\ \in\ A_{\vd}(X).
$$
\end{prop*}
 
This is a rewriting of parts of \cite[Chapter 14]{Fu} in the language of \cite{BF}. It is useful if we recognise a moduli space $M$ as having the same perfect obstruction theory as a $\wt D_r$. Then invariants defined by integration against $[M]^{\vir}$ can be calculated more easily on the smooth space $X$ \emph{if} the integrand can be expressed as a pullback from $X$.

The drawback is that the map $\wt D_r\to X$ contracts\footnote{\label{fn}Except when $D_r$ is the \emph{deepest} degeneracy locus --- i.e. when $D_{r+1}=\emptyset$. Then $\wt D_r\cong D_r$ and $\wt D_r\to X$ is an embedding. This is the case studied in \cite{GT1}.} the exceptional locus of $\wt D_r\to D_r$, so most integrands (such as the one in \eqref{expan} below) will not be pullbacks from $X$. We could instead work in $\Gr\;(r,E_0)$, to which $[\wt D_r]^{\vir}$ pushes forward to give
$c_{re_1}(\cU^*\otimes p^*E_1)$, but this description is not an invariant of the quasi-isomorphism class of  $E_\bullet$.

However, in examples like (\ref{**}, \ref{embedB}) below each kernel $\ker\sigma_x=h^0(E_\bullet|_x)$ naturally embeds in (the fibre over $x$ of) a certain vector bundle $B\to X$,
\beq{Bembed}
h^0(E_\bullet|_x)\ \subseteq\ B_x \quad\forall x\in X.
\eeq
The global condition is that $B^*$ surjects onto $\coker(\s^*)=h^0(E_\bu^\vee)$ over $D_r$ \eqref{Esurj}. Then $\wt D_r$ naturally embeds in $\Gr\;(r,B)/X$, and we can push forward the virtual class.

\begin{thm*} \label{main} Under the embedding
$\iota\colon\,\wt D_r(E_\bu ) \into \Gr\;(r,B)$
the pushforward of the virtual cycle to $A_{\vd}\big(\!\op{Gr}(r,B)\big)$is given by 
\beq{DeltaBX}
\iota_*\big[\wt D_r(E_\bullet)\big]^{\vir}\=\Delta^{r}_{\,\op{rk}(B)-\op{rk}(E_\bu)}\big(c\;(\cc Q- q^*E_\bu)\big),
\eeq
where $\cc Q$ is the universal quotient bundle over $q\colon\!\Gr\;(r,B)\to X$.
\end{thm*}

\medskip\noindent\textbf{Nested Hilbert schemes.}
For simplicity consider first \emph{2-step} Hilbert schemes
of nested subschemes (of dimensions 0 and 1) of a fixed smooth projective surface $S$. Given $\b\in H^2(S,\bb Z)$ and integers $n_1,n_2\ge0$ we set
\beq{Sn1n2b}
\S{n_1,n_2}_\b \ := \big\{I_1(-D)\subseteq I_2\subseteq\cO_S\ \colon\ [D]=\b,\ \mathrm{length}\;(\cO_S/I_i)=n_i \big\}.
\eeq
When  $n_1=0=n_2$ we use $S_\b=S_\b^{[0,0]}$ to denote the Hilbert scheme of divisors in class $\b$. Conversely when $\b=0$ we get the nested Hilbert scheme of points $\S{n_1,n_2}$ studied in \cite{GT1}; we review this briefly first.

\medskip\noindent\textbf{Nested Hilbert schemes of points.}
For simplicity we set $H^{\ge1}(\cO_S)=0$ for now.
For points $I_1,I_2\subset\cO_S$ of $S^{[n_1]},\,S^{[n_2]}$ we have
\beq{homI1I2}
\Hom(I_1,I_2)\=\left\{\!\!\begin{array}{ll} \C & I_1\subseteq I_2, \\
0 & I_1\not\subseteq I_2. \end{array}\right.
\eeq
Therefore $\iota\colon S^{[n_1,n_2]}\Into S^{[n_1]}\times S^{[n_2]}$ is the degeneracy locus $D_1(E_\bu)$ of the 2-term complex\footnote{This uses our simplifying assumption $H^{\ge1}(\cO_S)=0$. (For general $S$ we modify $E_\bu$ in \cite{GT1}, removing $H^{\ge1}(\cO_S)$ terms to make it 2-term.) The notation will be defined in Section \ref{nHs}; the $\cI_i$ are the universal ideal sheaves on $S\times S^{[n_1]}\times S^{[n_2]}$ and $\pi$ denotes the projection down $S$ to $S^{[n_1]}\times S^{[n_2]}$.} of vector bundles
$$
E_\bu\=R\hom_\pi(\cI_1,\cI_2)\quad\mathrm{over}\quad \S{n_1}\times \S{n_2}
$$
which, when restricted to the point $(I_1,I_2)$, has 0\;th cohomology  \eqref{homI1I2}. Moreover, \eqref{homI1I2} shows all higher degeneracy loci are empty, so we are in the situation of Footnote \ref{fn} with $r=1$. Thus the Thom-Porteous formula gives
$$
\iota_*\big[S^{[n_1,n_2]}\big]^{\vir}\=c_{n_1+n_2}\big(R\hom_\pi(\cI_1,\cI_2)[1]\big)\cap\big[S^{[n_1]}\times S^{[n_2]}\big].
$$
In \cite{GT1} we also showed the relevant perfect obstruction theory agrees (in K-theory at least) with that coming from Vafa-Witten theory. Therefore\footnote{Siebert's formula \cite[Theorem 4.6]{S} for the virtual class shows it depends on the perfect obstruction theory only through the K-theory class of its virtual tangent bundle.} Vafa-Witten invariants can be written as integrals over the smooth ambient space $S^{[n_1]}\times S^{[n_2]}$, as exploited to great effect in \cite{La1, La2} for instance.

\medskip\noindent\textbf{Nested Hilbert schemes of curves and points.}
When $\beta\ne0$ we consider
\beq{homI1I2D}
\Hom(I_1(-D),I_2) \quad\mathrm{for}\ D\mathrm{\ in\ class\ }\beta.
\eeq
In contrast to \eqref{homI1I2} this can become arbitrarily big, with different elements corresponding (up to scale) to different nested subschemes of $S$. Therefore, the corresponding nested Hilbert scheme $\S{n_1,n_2}_\b$ \emph{dominates} the degeneracy locus $D_1(E_\bu)$ of the complex of vector bundles
\beq{rhompi}
E_\bu\=R\hom_\pi(\cI_1(-\cD),\cI_2)\quad\mathrm{over}\quad \S{n_1}\times \S{n_2},
\eeq
which, when restricted to the point $(I_1,I_2,\cO(D))$, has 0\;th cohomology \eqref{homI1I2D}. Since a point of the nested Hilbert scheme $\S{n_1,n_2}_\b$ is a one dimensional subspace of the kernel \eqref{homI1I2D} of the complex \eqref{rhompi}, we see that
$$
\S{n_1,n_2}_\b\=\wt D_1(E_\bu)
$$
is the virtual resolution of the degeneracy locus $D_1(E_\bu)$ with $r=1$.
Since
\beq{**}
\Hom(I_1(-D),I_2)\ \subseteq\ H^0(\cO(D)),
\eeq
the $0$\;th cohomology of $E_\bu$ embeds in $H^0(\cO(D))$. So we get an example of \eqref{Bembed} with $r=1$ and $B$ the trivial vector bundle with fibre $H^0(\cO(D))$. The embedding $\wt D_1(E_\bu)\into\PP(B)$ of Theorem \ref{main} then becomes the following.

\begin{thm*}\label{ez} Suppose $H^{\ge1}(\cO_S)=0$ for now. Then $S^{[n_1,n_2]}_\b$ has a virtual cycle of dimension $n_1+n_2+h^0(\cO(D))-1$ which, under the embedding
$$
\iota\,\colon\ S^{[n_1,n_2]}_\b\Into S^{[n_1]}\times S^{[n_2]}\times\PP\big(H^0(\cO(D))\big),
$$
pushes forward to
$$
\iota_*\big[S^{[n_1,n_2]}_\b\big]^{\vir}\=
c\_{n_1+n_2}\Big(\!-\cO_{\PP(H^0(\cO(D)))}(-1)-R\hom_\pi(\cI_1,\cI_2(D))\Big).
$$
\end{thm*}

\medskip\noindent\textbf{Reduced virtual cycle.} For general $S$ with possibly nonzero $H^{\ge1}(\cO_S)$ we work instead with
\beq{rhompi2}
E_\bu\=R\hom_\pi(\cI_1(-\cL_\beta),\cI_2)\quad\mathrm{over}\quad X:=\S{n_1}\times \S{n_2}\times \Pic_\b(S),
\eeq
which, when restricted to the point $(I_1,I_2,\cO(D))$, has 0\;th cohomology \eqref{homI1I2D}. Here $\cL_\beta$ is a Poincar\'e line bundle over $S\times \Pic_\b(S)$, normalised (by tensoring by the pullback of $\cL_\b^{-1}|_{\{x\}\times\Pic_\b(S)}$ if necessary) so that $\cL_\b|_{\{x\}\times\Pic_\b(S)}$ is trivial on some fixed basepoint $x\in S$.

Again a point of $\S{n_1,n_2}_\b$ is a one dimensional subspace of the kernel \eqref{homI1I2D} of the complex \eqref{rhompi2}, so
$$
\S{n_1,n_2}_\b\=\wt D_1(E_\bu)
$$
is the virtual resolution of the degeneracy locus $D_1(E_\bu)$ with $r=1$. This gives a description of $S_\beta^{[n_1,n_2]}$ as a projective cone over $S^{[n_1]}\times S^{[n_2]}\times\Pic_\beta(S)$; see Proposition \ref{props}, generalising \cite[Lemma 2.15]{DKO}.

We can no longer take $B$ to have fibre $H^0(L)$ over $L\in\Pic_\b(S)$ since this may jump in dimension. Instead we fix a sufficiently ample divisor $A\subset S$ and let $B$ be the bundle with fibre $H^0(L(A))$:
$$
B\ :=\ \pi_*\big(\cL_\beta(A)\big) \qquad\mathrm{over\ }X.
$$
Therefore the inclusions
\beq{embedB}
\Hom(I_1(-L),I_2)\ \subseteq\ H^0(L)\ \subseteq\ H^0(L(A))
\eeq
give the required embeddings $h^0(E_\bu|_x)\subset B_x$ of \eqref{Bembed}. Thus $\wt D_1\subset\PP(B)$ such that $\cO_{\PP(B)}(-1)$ restricts to $\cU\to\wt D_1$, and Theorem \ref{main} gives the following.

\begin{thm*}\label{pushBX} Fix a surface $S$ and $\beta\in H^2(S,\Z)$.

$\bullet$ If $H^2(L)=0$ for all \emph{effective} $L\in\Pic_\beta(S)$ then $\S{n_1,n_2}_\b\cong\wt D_1(E_\bu)$ inherits a \emph{reduced} virtual cycle of dimension
$\chi(L)+n_1+n_2+h^1(\cO_S)-1$.

$\bullet$ If $H^2(L)=0$ for \emph{all} $L\in\Pic_\beta(S)$
then via $\iota\colon\S{n_1,n_2}_\b=\wt D_1(E_\bu)\into\PP(B)$ it pushes forward to
$$
\iota_{*}\big[\S{n_1,n_2}_\b\big]^{\op{red}}\=c\_{n_1+n_2+d}\Big(q^*B(1)-R\hom_\pi\big(\cI_1,\cI_2\otimes\cc L_\b(1)\big)\Big),
$$
where $q\colon \PP(B)\to \S{n_1}\times \S{n_2}\times \Pic_\b(S)$ is the projection, $\cO(1):=\cO_{\bb P(B)}(1)$ and $d=\frac12A.(2\b+A-K_S)=\chi(L(A))-\chi(L)$ for any $L\in\Pic_\beta(S)$.
\end{thm*}

\noindent The $H^2(L)=0$ condition ensures that $E_\bu$ is 2-term. (We handle $H^2(L)\ne0$ in the next Section.)
When $n_1=0$, this theorem was proven in \cite[Theorem A.7]{KT1}. Here --- as in \cite{KT1, GSY1} --- we call the virtual cycle the \emph{reduced} cycle $\big[\S{n_1,n_2}_\b\big]^{\mathrm{red}}$ since it differs from the one appearing in Vafa-Witten theory\footnote{Or indeed reduced DT theory, confusingly. Removing \emph{one} copy of $H^2(\cO_S)$ from the DT obstruction space of $\S{n_1,n_2}_\b$ gives the Vafa-Witten/reduced DT obstruction theory. Removing \emph{two} copies gives the degeneracy locus obstruction theory of Proposition \ref{mainp}.}  by an $H^2(\cO_S)$ term in the obstruction theory.

The apparent difference between the formulae in Theorems \ref{ez} and \ref{pushBX} will be explained by the formula \eqref{leo} in Theorem \ref{bar}.

\medskip\noindent\textbf{Non-reduced virtual cycle.} In general, when $H^2(L)$ need not vanish, $E_\bu$ \eqref{rhompi2} is a 3-term complex with nonzero $h^2(E_\bu)$ cohomology sheaf. So in order to apply our theory 
we use the \emph{splitting trick} from \cite[Section 6.1]{GT1}
to remove $R^2\pi_*\;\cO[-2]$ from $E_\bu(1)$ after pulling back to an affine bundle over $X$. (Since an affine bundle has the same Chow groups as its base, no information is lost from the virtual cycle.) 

We only manage this on a neighbourhood of $\S{n_1,n_2}_\b\subset \PP(B)$, making $E_\bu(1)$ a 2-term complex there. By Proposition \ref{mainp} this is enough to prove

\begin{thm*} This construction defines a virtual fundamental class
\beq{virclass}
\big[\S{n_1,n_2}_\b\big]^{\vir}\ \in\ A_{n_1+n_2+\vd(S_\beta)}\big(\S{n_1,n_2}_\b\big),
\eeq
where $\vd(S_\beta)=\frac12\beta.(\beta-K_S)$. It agrees with the virtual cycle of \cite{GSY1}.
\end{thm*}

\medskip\noindent\textbf{Computing virtual cycles by comparison.}
Since the above neighbourhood need not be all of $\PP(B)$, we get no useful push forward formula for \eqref{virclass} in general.\footnote{The exceptions are when $p_g(S)=0$ --- then \eqref{virclass} is the reduced virtual class of  Theorem \ref{pushBX} --- or when $H^2(L)=0$ for all effective $L\in\Pic_\b(S)$, so \eqref{virclass} is zero as the perfect obstruction theory which produced it can be reduced to give the class $\big[\S{n_1,n_2}_\b\big]^{\mathrm{red}}$.} This problem is down to the curve class $\beta$ alone --- if $\beta=0$ it does not arise, whereas if $n_1=0=n_2$ the problem already arises for the Hilbert scheme $S_\beta$ of pure curves in class $\beta$. However,
$$
[S_\beta]^{\vir}\,\in\,A_{\vd_\beta}(S_\beta), \qquad\vd_\b:=\vd(S_\beta)=\frac12\beta.(\beta-K_S),
$$
is a well-studied class that has been understood by other methods \cite{DKO,K} due to its importance in Seiberg-Witten theory. Separating out the parts of the obstruction theory governing the curve and the points, and applying the degeneracy locus technique to the latter only, we are able to prove the following comparison result. 

Let $\cD_\beta\subset S\times S_\beta$ denote the universal curve. We call
\beq{COdef}
\mathsf{CO}_\beta^{[n_1,n_2]}\ :=\ R\pi_*\;\cO(\cD_\b)-R\hom_\pi(\cI_1,\cI_2(\cD_\b))
\eeq
the Carlsson-Okounkov K-theory class \cite{CO} on $\S{n_1}\times \S{n_2}\times S_\b$.

\begin{thm*}\label{comparison} 
Pushing the virtual and reduced classes forwards along
$$
\iota \colon \S{n_1,n_2}_\b \Into\S{n_1}\times \S{n_2}\times S_\b\vspace{-2mm}
$$
gives
$$
\iota_*\big[\S{n_1,n_2}_\b\big]^{\vir}\=c_{n_1+n_2}\big(\mathsf{CO}_\beta^{[n_1,n_2]}\big) \cap\big[\S{n_1}\times \S{n_2}\big] \times  \big[S_\b\big]^{\vir}
$$
and, if $H^2(L)=0$ for all effective $L\in\Pic_\beta(S)$,
$$
\iota_*\big [\S{n_1,n_2}_\b\big]^{\op{red}}\=c_{n_1+n_2}\big(\mathsf{CO}_\beta^{[n_1,n_2]}\big) \cap \big[\S{n_1}\times\S{n_2}\big] \times \big[S_\b\big]^{\op{red}}.
$$
\end{thm*}

When $n_1=0$ the first of these results was proved in \cite{K} and the second follows from \cite[Appendix A]{KT1}.

\medskip\noindent\textbf{$\ell$-step Hilbert schemes.}
The generalization of Theorems \ref{pushBX} and  \ref{comparison} to $\ell$-step nested Hilbert schemes
$$
\S{n_1,\dots, n_\ell}_{\b_1,\dots,\b_{\ell-1}}
$$
is straightforward. See Section \ref{lstep} for details.

\medskip\noindent\textbf{Vafa-Witten invariants.} Vafa-Witten invariants \cite{TT1} are made up of ``instanton contributions" and ``monopole contributions". The former are virtual Euler characteristics of moduli spaces of rank $r$ sheaves on $S$, as studied in \cite{GK} for instance. The latter are integrals over moduli spaces of chains of sheaves on $S$ of total rank $r$ with nonzero maps between them.

When $p_g(S)>0$ and $r$ is prime a vanishing result \cite[Theorem 5.23]{TVW}
implies the only nonzero monopole contributions come from moduli spaces of chains of rank 1 sheaves. After tensoring with a line bundle, these are nested Hilbert schemes.
For instance in rank 2 the relevant integrals take the form
\beq{expan}
\int_{\big [\S{n_1,n_2}_\b\big]^{\vir}}\sum_i \rho^* \a_i \cup h^i\=
\sum_i\int_{\rho_*\big(h^i\cap\big [\S{n_1,n_2}_\b\big]^{\vir}\big)}\a_i,
\eeq
where $\a_i\in H^*(\S{n_1}\times \S{n_2}\times \Pic_\b(S))$ and $h:=c_1\big (\cO_{\bb P(B)}(1)\big)$ and
$$
\rho\ \colon\ \S{n_1,n_2}_\b \To S^{[n_1]}\times S^{[n_2]}\times \Pic_\b(S).
$$
Using Theorem \ref{comparison} we can express this in terms of (a) integrals over the smooth space $S^{[n_1]}\times S^{[n_2]}\times \Pic_\b(S)$, and (b) integrals over $S_\beta$. The latter give the Seiberg-Witten invariants \cite{DKO, CK},
$$
\mathsf{SW}_\b\ :=\ \left\{\!\!\begin{array}{lll} \deg\,[S_\b]^{\vir} && \vd_\b=\frac12\beta.(K_S-\beta)=0, \\ 0 && \mathrm{otherwise.}\end{array}\right..
$$

\begin{thm*} \label{satis} Suppose $p_g(S)>0$. Then in $H_*\big(S^{[n_1]}\times S^{[n_2]}\times\Pic_\b(S)\big)$,
\begin{align*}
\rho_*\big [\S{n_1,n_2}_\b\big]^{\vir}&\=\mathsf{SW}_\beta\cdot c_{n_1+n_2}\big(\!-\!R\hom_\pi(\cI_1,\cI_2\otimes L)\big)\times[L], \\
\rho_*\Big(\;\!h^i\cap\big [\S{n_1,n_2}_\b\big]^{\vir}\Big)&\=0 \quad\mathrm{for\ }i>0,
\end{align*}
where $L\in\Pic_\b(S)$ so $[L]$ is the generator of $H_0(\Pic_\b(S),\Z)$.
\end{thm*}

When $p_g(S)=0$ then $S$ may not be of ``Seiberg-Witten simple type", which means we have to consider higher dimensional Seiberg-Witten moduli spaces. Letting $\mathsf{AJ}\colon S_\b\to\Pic_\b(S)$ denote the Abel-Jacobi map, the higher Seiberg-Witten invariants \cite{DKO} are
$$
\mathsf{SW}_\b^{\;j}\ :=\ \mathsf{AJ}_*\big(h^j \;\cap \; [S_\b]^{\vir}\big)\ \in\ H_{2(\vd_\b-j)}(\Pic_\b(S))\=\wwedge^{\!\;2(\vd_\b-j)}H^1(S).
$$

\begin{thm*} \label{notsatis}
If $p_g(S)=0$ then in $H_{2(n_1+n_2+\vd_\b-i)}(S^{[n_1]}\times S^{[n_2]}\times\Pic_\b(S))$,
\begin{multline*}
\bullet\ \rho_*\big(h^i\cap\big[\S{n_1,n_2}_\b\big]^{\vir}\big)\= \\
\sum_{j=0}^{n_1+n_2}c_{n_1+n_2-j}\big(R\pi_*\;\cL_\b-R\hom_\pi(\cI_1,\cI_2\otimes\cc L_\b)\big)\cup\mathsf{SW}_\b^{\;i+j}.\qedhere
\end{multline*}

$\bullet$ If $H^2(L)=0\ \forall\,L\in\Pic_\b(S)$ then, setting $d=n_1+n_2+h^1(\cO_S)-\vd_\b$,
$$
\rho_{*} \big(h^i \cap\big [\S{n_1,n_2}_\b\big]^{\vir}\big)\=c_{d+i}\big(\!-R\hom_\pi(\cI_1,\cI_2\otimes\cL_\b)\big).
$$

$\bullet$ If not,
$$
\rho_{*}\big(h^i \cap\big[\S{n_1,n_2}_{\b}\big]^{\vir}\big)\=\mathsf{SW}_\b\cdot c_{n_1+n_2+i}\big(\!-R\hom_\pi(\cI_1,\cI_2\otimes L)\big)\times[L],
$$
and both sides vanish when $i>0$. (Cf. Theorem \ref{satis}.)
\end{thm*}

Now set $\b^{\vee}:=K_S-\b$, so that the condition that $H^2(L)\ne0$ for some $L\in\Pic_\b(S)$ is the condition that $\beta^\vee$ is effective. 
Seiberg-Witten theory has a duality under $\beta\leftrightarrow\beta^\vee$. This gives interesting dualities between invariants of nested Hilbert schemes under $\beta\leftrightarrow\beta^\vee,\ n_1\leftrightarrow n_2$. Define the map
$$
\rho^\vee\,\colon\ \S{n_2,n_1}_{\b^\vee}\To S^{[n_1]}\times S^{[n_2]}\times \Pic_\b(S)
$$
by replacing $\beta\leftrightarrow\beta^\vee,\ n_1\leftrightarrow n_2$ in
$\rho\colon\S{n_1,n_2}_\b\To S^{[n_1]}\times S^{[n_2]}\times \Pic_\b(S)$
and then composing with $L\mapsto K_S\otimes L^{-1}\colon$ $\Pic_{\b^\vee}(S)\to\Pic_\b(S)$.

\begin{thm*}[Duality]\label{dualite'} In $H_{2(n_1+n_2+\vd_\b-i)}(S^{[n_1]}\times S^{[n_2]}\times\Pic_\b(S))$,
\begin{align*}
\bullet& \text{ if }p_g(S)>0,\quad%
\rho_*\big(h^i\cap\big[\S{n_1,n_2}_\b\big]^{\vir}\big)\=(-1)^{s+i}\,\rho^\vee_{*}\big(h^i\cap\big[\S{n_2,n_1}_{\b^\vee}\big]^{\vir}\big), \\
\bullet& \text{ if }p_g(S)=0, \quad\,
\rho_{*}\big(h^i\cap\big[\S{n_1,n_2}_{\b}\big]^{\vir}\big)\=(-1)^{s+i}\;\rho^\vee_{*}\big(h^i\cap\big[\S{n_2,n_1}_{\b^\vee}\big]^{\vir}\big)+ \\
&\hspace{7cm} c_{d+i}\big(\!-R\hom_\pi(\cI_1,\cI_2\otimes\cL_{\b})\big),
\end{align*}
where $s=n_1+n_2-\chi(\cO_S)-\vd_\b=d-1$.
\end{thm*}

\begin{thm*}
If $p_g(S)>0$ the contributions of the nested Hilbert schemes $\S{n_1,\dots, n_\ell}_{\b_1,\dots,\b_{\ell-1}}$ to the Vafa-Witten invariants of $S$ are invariants of its oriented diffeomorphism type.
\end{thm*}

When $p_g(S)=0$ this can fail because only the unordered pair $(\mathsf{SW}_\b,\mathsf{SW}_{\b^\vee})$ is invariant under oriented diffeomorphisms, rather than the individual invariant $\mathsf{SW}_\b$.

\medskip\noindent\textbf{Laarakker.}
In \cite{La1} (stable case) and \cite{La2} (general case) Laarakker uses these results to compute Vafa-Witten invariants \cite{TT1,TT2} and \emph{refined} Vafa-Witten invariants \cite{TVW} on surfaces with $p_g>0$ (and $h^{0,1}=0$ for now).

By the comparison result for virtual cycles of Theorem \ref{comparison} the contributions from points and curves split, with the curves contributing Seiberg-Witten invariants --- i.e. well-understood integrals over linear systems of curves. Laarakker evaluates the contributions of Hilbert schemes of points via the method of \cite{EGL}. The result depends only on the curve class $\beta\in H_2(S,\Z)$ and the cobordism class of the surface --- and thus only on $c_1(S)^2,\,c_2(S)$ and $\beta^2$. Therefore these contributions can be calculated on K3 surfaces and toric surfaces (despite these not having $p_g>0$!).

The results are very general, but the absolute simplest to state is for $S$ minimal of general type with $p_g(S)>0$  and $H_1(S,\Z)=0$, such that $K_S$ is not divisible by 2.

\begin{thm*}\cite{La1}
Let $S$ be as above. The monopole branch contributions\footnote{See \cite{TVW} for definitions. The monopole branch contributions come from $\C^*$-fixed Higgs pairs $(E,\phi)$ on $S$ with $\det E=K_S,\,\tr\phi=0$ and $\phi\ne0$.} of rank 2 Higgs pairs with $\det=K_S$ to the refined Vafa-Witten generating series $\sum_n\mathsf{VW}_{2,K_S,n}(t)\;q^n$ can be written
$$  
A(t,q)^{\chi(\cO_S)}B(t,q)^{c_1(S)^2},
$$
where
$$
A(t,q),\ B(t,q)\ \in\ \Q(t^{1/2})(\!(q^{1/2})\!)
$$
are universal functions, independent of $S$. 
\end{thm*}

Furthermore, K3 calculations \cite{GK, TVW} determine $A(t,q)$ completely in terms of modular forms. And $B(t,q)$ can be predicted by modularity and the results of G\"ottsche-Kool \cite{GK} on the instanton contributions; he checks this prediction in low degree by toric computations.\medskip

\smallskip\noindent\textbf{Acknowledgements.} We thank Martijn Kool and Ties Laarakker for help and useful conversations, Bhargav Bhatt for pointing out the Jouanolou trick, and two thorough referees for useful suggestions. A.G. acknowledges partial support from NSF grant DMS-1406788. R.T. is partially supported by EPSRC grant EP/R013349/1.

\medskip\noindent\textbf{Notation.} Given a map $f\colon X\to Y$, we often use the same letter $f$ to denote its basechange by any map $Z\to Y$, i.e. $f\colon X\times_YZ\to Z$. In particular projections $S\times M\to M$ down the surface $S$ are all denoted by $\pi$. We often suppress pullback maps $f^*$ on sheaves when it shouldn't cause confusion. 

For a vector bundle $B$ over $X$, we denote by $\PP(B)$ the projective bundle of lines in $B$. For a coherent sheaf $\cc F$ on $X$, we denote by $\PP^*(\cc F)=\Proj\Sym^\bu\!\cF$ the projective cone of quotient lines of $\cc F$. When $\cF$ is locally free then of course $\PP^*(\cF)\cong\PP(\cF^*)$.

\subsection{Splitting trick}\label{Jtrick} Since we are ultimately interested in Chern class formulae, we often only care about the topological type of a given algebraic vector bundle. In particular, we will often wish to split nontrivial extensions of vector bundles, which does not change their topological type. One way to do this is by working with $C^\infty$ bundles; we prefer to stay within algebraic geometry and use the \emph{Jouanolou trick} \cite{J} instead. 

Namely, if $Y$ is quasi-projective, there exists an \emph{affine bundle} $\rho\colon\wt Y \to Y$ such that $\wt Y$ is an affine variety. 
Therefore, on pulling locally free sheaves back by $\rho^*$, they become \emph{projective} $\cO_{\wt Y}$-modules and any extensions become split. But $\rho$ is a homotopy equivalence, inducing an isomorphism of Chow groups [Kr, Corollary 2.5.7]
\beq{J=}
\rho^* \colon A_{*} (Y) \rt\sim A_{*+\dim\rho\,}\big(\wt Y\big),
\eeq
so formulae for cycles upstairs on $\wt Y$ induce similar formulae on $Y$.

A very explicit such affine bundle was used in \cite[Section 6.1]{GT1}, and all the necessary pullbacks were laboriously documented. In this paper, for the sake of clarity and economy of notation, we will often just say ``\emph{by the Jouanolou trick}" and give no further details. So we will pretend we are working on $Y$, suppressing mention of the fact that we have replaced $Y$ by its homotopic affine model $\wt Y$ and then used the isomorphism \eqref{J=} to recover results on $Y$ at the end.

\section{Virtual resolutions of degeneracy loci}\label{2}
Fix $X$ a smooth complex quasi-projective variety and
\beq{2term2}
\s\,\colon\ E_0\To E_1
\eeq
a map of vector bundles of ranks $e_0$ and $e_1$ over $X$. For each positive integer $r\le e_0$, we let 
$$
D_r(\s)\ :=\ \big\{ x\in X \colon\dim\ker\;(\s_x) \ge r \big\}
$$
denote the $r$\;th degeneracy locus of $\s$, with scheme structure defined by the vanishing of
\beq{wwe}
\wwedge^{e_0-r+1}\s\,\colon\ \wwedge^{e_0-r+1} E_0\To \wwedge^{e_0-r+1} E_1.
\eeq
This is the $(e_1-e_0+r-1)$th Fitting scheme \cite[Section 20.2]{Ei} of the first cohomology sheaf $h^1(E_\bu)=\coker\sigma$ of the complex $E_\bu$ \eqref{2term2}. In particular it depends on $E_\bu$ only up to quasi-isomorphism, and we often denote it $D_r(E_\bu)$.
Then define
\beq{Grdef}
\wt D_r(E_\bu)\=\wt D_r(\s)\ :=\,\Gr\;(\coker \s^*,r)
\eeq
to be the relative Grassmannian of $r$ dimensional quotients of the cokernel sheaf $h^0(E_\bu^\vee)$ of $\s^*\colon E_1^*\to E_0^*$. This is the moduli space of $r$ dimensional subspaces of fibres $(E_0)_x$ which are annihilated by $\sigma$, in the following sense.

Consider the functor
\begin{eqnarray}\nonumber
\mathrm{Schemes}/X &\To& \mathrm{Sets,} \\
\hspace{15mm} (f\colon S\to X) &\Mapsto& \big\{\mathrm{rank\ }r\mathrm{\ subbundles\ }U\subseteq f^*E_0\mathrm{\ which\ are}
\nonumber \\ \label{funct} && \hspace{14mm}
\mathrm{subsheaves\ of\ }h^0(f^*E_\bu)\subseteq f^*E_0\big\}.
\end{eqnarray}
Let $p\colon\Gr\;(r,E_0)\to X$ denote the relative Grassmannian, with universal subbundle $\cU\subset p^*E_0$.

\begin{prop} \label{pcut}
The functor \eqref{funct} is represented by $\wt D_r(\s)/X$. It embeds
\beq{iota1}
\iota\,\colon\,\wt D_r\ \Into\ \Gr\;(r,E_0)
\eeq
as the zero locus of \,$\wt\sigma\in\Gamma(\cU^*\otimes p^*E_1)$ defined by the composition
\beq{0loc}
\cU\Into p^*E_0\rt{p^*\s}p^*E_1.
\eeq
Finally, the projection $p\colon\wt D_r\to X$ has scheme theoretic image inside $D_r$. 
\end{prop}

\begin{rmks} Clearly $p\colon\wt D_r(\s)\to D_r(\s)$ is a set-theoretic surjection --- the fibre over $x\in D_r(\s)$ is $\Gr\;(r,\ker\sigma_x)$ --- but it need not be onto as a map of schemes. For instance $x.\id\colon\cO^{\oplus2}\to\cO^{\oplus2}$ on $\C_x$ has $D_1=(x^2=0)\subset\C_x$, the thickened origin, whereas $\wt D_1$ is a reduced $\PP^1$ over the origin.

By the definition of the functor \eqref{funct}, $\wt D_r(\s)$ comes equipped with a universal subbundle $\cU\subseteq p^*E_0$. It is the restriction of the universal subbundle on $\Gr\;(r,E_0)$, and the dual of the universal quotient bundle on $\Gr\;(\coker(\s^*),r)$. Its fibre over the point $\big(x,U\subseteq\ker\;(\s_x)\big)$ is $U$.
\end{rmks}

\begin{proof}
Since $\Gr\;(\mathrm{coker}\,\s^*,r)/X$ represents the functor
\begin{eqnarray}\nonumber
\mathrm{Schemes}/X &\To& \mathrm{Sets,} \\
\hspace{15mm} (f\colon S\to X) &\Mapsto& \big\{\mathrm{rank\ }r\mathrm{\ locally\ free\ quotients\ of\ coker\,}f^*\sigma^*\big\} \label{funct2}
\end{eqnarray}
it is sufficient to prove this is isomorphic to \eqref{funct}. A locally free quotient $Q$ of coker$\,(f^*\s^*)$ gives a complex
\beq{Qeq}
f^*E_1^*\rt{f^*\s^*}f^*E_0^*\To Q\To0
\eeq
which is exact at $Q$. Since $Q$ is locally free, the dual complex
\beq{Ueq}
0\To U\To f^*E_0\rt{f^*\s}f^*E_1,
\eeq
is exact at $U:=Q^*$ \emph{after any basechange} --- thus $U\into f^*E_0$ is a subbundle, giving an element of the set \eqref{funct}. The converse is easier: any complex \eqref{Ueq} with $U$ a subbundle dualises to a complex \eqref{Qeq} exact at $Q:=U^*$.

%
%
\medskip

The functor \eqref{funct} is a subfunctor of the functor taking $f\colon S\to X$ to the set $\{$rank $r$ subbundles $U\subset f^*E_0\}$ represented by $\Gr\;(r,E_0)/X$. This induces an embedding of $X$-schemes $\wt D_r(\s)\into\Gr\;(r,E_0)$ under which the universal subbundle $\cU$ pulls back to the dual $\cQ^*$ of the universal quotient bundle. Its image is contained in the zero locus $Z$ of \eqref{0loc} by applying \eqref{funct} to $S=\wt D_r$. So $\wt D_r(\s)\subseteq Z$.

Conversely, setting $f\colon Z\to X$ to be the projection, the  restriction of \eqref{0loc} to $Z$ gives
a rank $r$ subbundle $\cU|_Z\subset f^*E_0$ which factors through $\ker\;(f^*\s)$ by the definition of $Z$ as the zero locus. This gives an element of the set \eqref{funct}, classified by a map $Z\to \wt D_r(\s)$. The two maps are mutual inverses by construction.

Finally $p^*\s$ factors through $p^*E_0/\cU$  over $Z=\wt D_r(\s)$ by \eqref{0loc}. Therefore
$p^*\wwedge^{e_0-r+1}\s$ \eqref{wwe} factors through 
$\wwedge^{e_0-r+1}p^*E_0/\cU$. But this is zero because $\rk\;(p^*E_0/\cU)=e_0-r$, so $\wt D_r(\s)$ projects into the zero locus $D_r$ of \eqref{wwe}.
\end{proof}

\noindent\textbf{Perfect obstruction theory.} Let $I\subset\cO_{\Gr\;(r,E_0)}$ denote the ideal of $\wt D_r$ generated by $\wt\s$ \eqref{0loc}. Then $\wt D_r$ inherits the standard perfect obstruction theory
\beq{potbarn}
\xymatrix@R=5pt@C=14pt{
\cU\otimes p^*E_1^*|_{\wt D_r} \ar[dd]_{\wt\s} \ar[rr]^-{d\;\widetilde\sigma|_{\wt D_r}}&& \Omega_{\Gr}|_{\wt D_r} \ar@{=}[dd] &&& \LL^{\vir}_{\wt D_r} \ar[dd] \\ &&& \ar@{=}[r]&  \\
I/I^2 \ar[rr]^d&& \Omega_{\Gr}|_{\wt D_r} &&& \LL_{\wt D_r}.\!\!}
\eeq
That is, the complex on the bottom row is a representative of the truncated cotangent complex $\LL_{\wt D_r}$ and the top row is the virtual cotangent bundle of the obstruction theory.

\begin{prop}\label{vcb}
The virtual cotangent bundle $\LL^{\vir}_{\wt D_r}$ is a cone
$$
\Cone\Big\{\cU\otimes p^*E_\bullet^\vee\big|_{\wt D_r}[-1]\To\big(\cU^*\!\otimes\cU\;[-1]\ \oplus\ p^*\Omega_X\big)\big|_{\wt D_r}\Big\}
$$
of virtual dimension $\vd:=\dim X-r(e_1-e_0+r)$. It depends only on the quasi-isomorphism class of the complex $E_\bullet$ \eqref{2term}. The pushforward of the resulting virtual cycle
$$
\big[\wt D_r(E_\bullet)\big]^{\vir}\ \in\ A_{\vd}\big(\wt D_r(E_\bullet)\big)
$$
to $X$ is given by the Thom-Porteous formula
$$
\Delta\;_{r-\op{rk}(E_\bu)}^r\big(c(-E_\bullet)\big)\ \in\ A_{\vd}(X).
$$
\end{prop}

\begin{proof}
Modifying $\LL^{\vir}_{\wt D_r}$ \eqref{potbarn} by an acyclic complex, it is quasi-isomorphic to the total complex
\beq{mod3}
\xymatrix@R=18pt@C=40pt{
\cU\otimes p^*E_1^*|_{\wt D_r} \ar[d]_{\id_{\cU\!}\otimes\!}^{\!p^*\sigma^*} \ar[r]^-{d\;\widetilde\sigma|_{\wt D_r}}& \Omega_{\Gr}|_{\wt D_r} \ar[d]^\phi \\
\cU\otimes p^*E_0^*|_{\wt D_r} \ar@{=}[r]& \cU\otimes p^*E_0^*|_{\wt D_r},\!\!}
\eeq
where $\phi$ is the composition $\Omega_{\Gr}\cong\cU\otimes(p^*E_0/\cU)^*\to\cU\otimes p^*E_0^*$. This diagram was shown to commute in \cite[Claim 2, proof of Theorem 3.6]{GT1}.

Therefore $\LL^{\vir}_{\wt D_r}\cong\Cone\!\big(\cU\otimes p^*E_\bu^\vee|_{\wt D_r}\To\Cone(\phi)\big)[-1]$. To determine $\Cone(\phi)$ we fit $\phi$ into the following commutative diagram, all of whose rows and columns are exact triangles:
\beq{tris}
\xymatrix@R=16pt@C=20pt{
p^*\Omega_X|_{\wt D_r} \ar[r]& \Omega_{\Gr}|_{\wt D_r} \ar[r]\ar[d]_\phi& \Omega_{\Gr\!/X}|_{\wt D_r} \ar[d] \\
& \cU\otimes p^*E_0^*|_{\wt D_r} \ar[d]\ar@{=}[r]& \cU\otimes p^*E_0^*|_{\wt D_r} \ar[d] \\
& \Cone(\phi) \ar[r]& \cU\otimes\cU^*|_{\wt D_r} \ar[r]& p^*\Omega_X|_{\wt D_r}[2].\!}
\eeq
We calculate the lower right hand arrow. By definition it is the composition from the top left to the bottom right of
\beq{tri2}
\xymatrix@R=14pt{
\cU\otimes\cU^*|_{\wt D_r} \ar[r]& \Omega_{\Gr\!/X}|_{\wt D_r}[1] \ar@{=}[d]\ar[r]& \cU\otimes p^*E_0^*|_{\wt D_r}[1] \\
\Omega_{\Gr}|_{\wt D_r}[1] \ar[r]& \Omega_{\Gr\!/X}|_{\wt D_r}[1] \ar[r]& p^*\Omega_X|_{\wt D_r}[2].\!}
\eeq
Here the upper row is the rightmost vertical exact triangle in \eqref{tris}, while the lower row is the exact triangle along the top row of \eqref{tris}.
The first arrow in \eqref{tri2} is well known to be the relative (to $X$) Atiyah class of the bundle $\cU$. So using its absolute Atiyah class we get the commutative diagram
$$
\xymatrix@R=16pt@C=50pt{
\cU\otimes\cU^*|_{\wt D_r} \ar[r]^-{\At_{\Gr\!/X}(\cU)}\ar[d]_{\At_{\Gr}(\cU)}& \Omega_{\Gr\!/X}|_{\wt D_r}[1] \ar@{=}[d]\ar[r]& \cU\otimes p^*E_0^*|_{\wt D_r}[1] \\
\Omega_{\Gr}|_{\wt D_r}[1] \ar[r]& \Omega_{\Gr\!/X}|_{\wt D_r}[1] \ar[r]& p^*\Omega_X|_{\wt D_r}[2]}
$$
which shows the composition $\cU\otimes\cU^*\to p^*\Omega_X[2]$ is zero. Therefore the bottom row of \eqref{tris} gives $\Cone(\phi)\cong\cU\otimes\cU^*\oplus p^*\Omega_X|_{\wt D_r}[1]$, implying the claimed result. \medskip

Finally, since the perfect obstruction theory arises from $\wt D_r$'s description as the zero locus of a section of the bundle $\cU^*\otimes p^*E_1\to\Gr$, the resulting virtual cycle is the localised top Chern class \cite[Section 14.1]{Fu} of that bundle. Its pushforward to $\Gr$ is
\beq{cUE}
c_{\;re_1}(\cU^*\otimes p^*E_1);
\eeq
further pushing down $p\colon\Gr\to X$ shows the pushforward of $[\wt D_r]^{\vir}$ to $X$ is given by 
$$
p_*\;c_{\;r\;e_1}(\cU^*\otimes p^*E_1)\=\Delta\;_{e_1-e_0+r}^r\big(c(-E_\bullet)\big)
$$
by \cite[Theorem 14.4]{Fu}.
\end{proof}

This proves Proposition \ref{mainp} in the Introduction. But in passing from $\Gr\;(r,E_0)$ down to $X$ we have contracted $\wt D_r$ back to $D_r$  and so lost information in general. (The exception is the case studied in \cite{GT1} where $r$ is the largest integer for which $D_r$ is nonempty; then $p|_{\wt D_r}$ is an embedding.)

We could work in $\Gr\;(r,E_0)$, in which the class of the virtual cycle is $c_{\;r\;e_1}(\cU^*\otimes p^*E_1)$, but this is not a quasi-isomorphism invariant of the complex $E_\bu$. In examples such as (\ref{**}, \ref{embedB}) we can replace $E_0$ by a more canonical bundle $B$, so now we assume we have chosen $B$ such that
\beq{inx}
h^0(E_\bu|_x)\Into B_x \quad\text{for each }x\in D_r.
\eeq
For basechange reasons the precise condition is most easily stated via duals. \smallskip

\noindent\textbf{Choice.} \emph{We choose a vector bundle $B$ on $X$ and a surjection} 
\beq{Esurj}
B^*\big|_{D_r}\To h^0\big(E_\bu^\vee\big)\big|_{D_r}\To0.
\eeq
\emph{For some purposes we will require the surjection \eqref{Esurj} extends to all of $X$,}
\beq{Esurj2}
B^*\To h^0\big(E_\bu^\vee\big)\To0.
\eeq
\emph{For instance if $r=1$ this is immediate from \eqref{Esurj}.}\medskip

Note that by restricting both exact sequences $E_1^*\to E_0^*\to h^0(E_\bu^\vee)\to0$ and \eqref{Esurj} to $x\in D_r$ and dualising we recover \eqref{inx}.

\medskip
The surjection $B^*|_{D_r}\to(\coker\s^*)|_{D_r}$ of \eqref{Esurj} induces an embedding $\Gr\;(\coker\s^*,r)\subset\Gr\;(B^*,r)$ since $\Gr\;(\coker\s^*,r)$ lies over $D_r\subset X$. By Proposition \ref{pcut} this is an embedding
\beq{iot2}
\iota\_{B}\,\colon\ \wt D_r\Into\Gr\;(r,B)\rt{q}X
\eeq
such that $\iota_B^*\;\cU_B\cong\iota^*\cU$. Here $\cU_B$ and $\cQ_B=q^*B/\cU_B$ denote the universal sub- and quotient bundles on $\Gr\;(r,B)$, and $\iota\,\colon\,\wt D_r\into\Gr\;(r,E_0)$ is the embedding \eqref{iota1}. We will compare these via a third embedding,
$$
I:=\iota\times\_X\iota\_{B}\ \colon\ \wt D_r\Into
\Gr\;(r,E_0)\times\_X\Gr\;(r,B),
$$
fitting into the Cartesian diagram
\beq{dgr}
\xymatrix@R=20pt{
\wt D_r\times\_X\Gr\;(r,B) \INTO^-{\iota\times1}\ar[d]_q& \Gr\;(r,E_0)\times\_X\Gr\;(r,B) \ar[d]^{q\_0} \\
\wt D_r\! \ar[ur]-<15pt,13pt><1ex>^(.45)I\ar@/^{-2ex}/[u]_s \INTO^(.45){\iota}& \Gr\;(r,E_0).}
\eeq
Here the vertical maps are flat and the section $s:=\id_{\wt D_r}\!\times\_{\!X\,}\iota\_{B}$. As usual we suppress various pullback or basechange maps.

\begin{lem} \label{s*s*} The section $s$ in \eqref{dgr} is a regular embedding, with image cut out by the regular section of $\cU^*\!\otimes\cQ_B$ defined by the composition
\beq{compU}
\cU\Into B\To\hspace{-5.5mm}\To\cQ_B \qquad\mathrm{on}\quad \wt D_r\times\_X\Gr\;(r,B).
\eeq
\end{lem}

\begin{proof}
Since $s^*q^*\iota^*\cU=\iota^*\cU=s^*\cU_B$ on $\wt D_r\times\_X\Gr\;(r,B)$ we see that $s^*$ applied to \eqref{compU} gives the same as $s^*$ applied instead to the composition
$$
\cU_B\Into B\To\hspace{-5.5mm}\To \cQ_B \qquad\mathrm{on}\quad \wt D_r\times\_X\Gr\;(r,B).
$$
This is zero by the definition of $\cQ_B$, so $s(\wt D_r)$ lies in the zero locus $Z$ of \eqref{compU}. We are left with showing that $Z\subseteq  s(\wt D_r)$.

By the definition of $Z$, pulling back the composition \eqref{compU} by $\iota\_Z\colon Z\into\wt D_r\times\_X\Gr\;(r,B)$ shows the composition along the top row of the following diagram is zero. This uniquely fills in the dotted arrow $\phi$ to the lower short exact sequence.
$$
\xymatrix@R=16pt{
0 \ar[r]& \iota_Z^*\cU \ar[r]\ar@{..>}[d]_\phi& \iota_Z^*B \ar[r]\ar@{=}[d]& \iota_Z^*\cQ_B \ar[r]& 0 \\
0 \ar[r]& \iota_Z^*\cU_B \ar[r]& \iota_Z^*B \ar[r]& \iota_Z^*\cQ_B \ar[r]& 0.}
$$
Since $\phi$ is a vector bundle injection it is an isomorphism. And since $\cU=\iota_B^*\;\cU_B=q^*s^*\cU_B$ is pulled back from $\wt D_r$, this gives
$$
\xymatrix@R=16pt{
\iota_Z^*\;q^*s^*\cU_B \INTO\ar@{=}[d]_{\phi\;}& \iota_Z^*B \ar@{=}[d] \\
\iota_Z^*\;\cU_B \INTO& \iota_Z^*B.\!}
$$
Thus this subbundle is represented by both of the maps $\iota\_Z$ and $s\circ q\circ\iota\_Z$, but any subbundle is classified by a unique map $Z\to\wt D_r\times\_X\Gr\;(r,B)$. So $\iota\_Z=s\circ q\circ\iota\_Z$, which gives $Z\subseteq s(\wt D_r)$.

Since codim\,$(s(\wt D_r))=\dim\Gr\;(r,B)-\dim X=\rk\,(\cU^*\!\otimes\!\cQ_B)$ it follows that $s$ is a regular embedding.
\end{proof}

\begin{thm} \label{bar} Let $b:=\rk B$. The pushforward of $[\wt D_r]^{\vir}$ to $\Gr\;(r,B)$ is given by 
$$
\iota\_{B*}[\wt D_r]^{\vir}\ =\ 
\Delta^{r}_{b+e_1-e_0}\big(c\big(\cQ_B-E_\bu\big)\big).
$$
If \eqref{Esurj2} holds this also equals
\beq{leo}
c\_{r(b-e_0+e_1)}\big(\cU_B^*\;\otimes(B-E_\bu)\big).
\eeq
\end{thm}


\begin{proof}
By Proposition \ref{pcut} and flat basechange around the diagram \eqref{dgr},
$$
(\iota\times1)_*\;q^*[\wt D_r]^{\vir}\=q_0^*\,\iota\_{*}[\wt D_r]^{\vir}\=q_0^*\;c\_{re_1}(\cU^*\otimes E_1).
$$
Therefore Lemma \ref{s*s*} gives
\beqa
I_*[\wt D_r]^{\vir} &=& (\iota\times1)_*s_*[\wt D_r]^{\vir}
\\ &=& (\iota\times1)_*s_*s^*q^*[\wt D_r]^{\vir} \\ &=&
(\iota\times1)_*\big(c\_{r(b-r)}(q^*\cU^*\otimes\cQ_B)\cap q^*[\wt D_r]^{\vir}\big) \\
&=& c\_{r(b-r)}(q_0^*\cU^*\otimes\cQ_B)\cap(\iota\times1)_*\;q^*[\wt D_r]^{\vir} \\
&=& c\_{r(b-r)}(q_0^*\cU^*\otimes\cQ_B)\cap c\_{re_1}(q_0^*\cU^*\otimes E_1) \\
&=& c\_{r(b-r+e_1)}\big(q_0^*\cU^*\otimes(\cQ_B\oplus E_1)\big).
\eeqa
Pushing down $p\colon\Gr\;(r,E_0)\times\_X\Gr\;(r,B)\to\Gr\;(r,B)$ gives 
\beq{toob}
\iota\_{B*}[\wt D_r]^{\vir}\=\Delta^{r}_{b+e_1-e_0}\big(c\big(\cQ_B+E_1-E_0\big)\big)
\eeq
by \cite[Proposition 14.2.2]{Fu}.

To prove \eqref{leo} we need to lift the surjection $B^*\to\coker(\s^*)$ of \eqref{Esurj2} to a map $\psi\colon B^*\to E_0^*$; see \eqref{Lift} below for how to do this after using the Jouanolou trick of Section \ref{Jtrick}. This then gives $h^0(E_\bu|_x)\into B_x$ for all $x\in X$, which means that
$$
E_0\rt{(\s,\psi^*)}E_1\oplus B
$$
is a subbundle. Therefore its cokernel $C$ is a rank $b+e_1-e_0$ \emph{vector bundle} over $X$, equivalent in K-theory to $B-E_\bu$. Thus we can apply \cite[Example 14.4.12]{Fu} to rewrite \eqref{toob} as
\beqa
\iota\_{B*}[\wt D_r]^{\vir} &=& \Delta^{r}_{b+e_1-e_0}\big(c\big(C-\cU_B\big)\big) \\
&=& c\_{r(b+e_1-e_0)}\big(\cU_B^*\;\otimes C\big) \\
&=& c\_{r(b+e_1-e_0)}\big(\cU_B^*\;\otimes(B-E_\bu)\big). 
\eeqa
This formula should be compared with \eqref{cUE}.
\end{proof}


\subsection*{Digression: description as a deepest degeneracy locus}
As usual fix a smooth quasi-projective variety $X$ carrying
a 2-term complex of bundles
$$
\sigma\colon E_0\To E_1.
$$
Almost by definition the virtual resolution $\wt D_r$ of the $r$th degeneracy locus can be seen as the \emph{deepest degeneracy locus} of the complex $\cU\to p^*E_1$ \eqref{vanloc} up on $\Gr\;(r,E_0)$. Such deepest degeneracy loci (those for which $r$ is maximal --- i.e. $D_{r+1}$ is empty) were studied in \cite{GT1}.

Since $E_0$ is not a quasi-isomorphism invariant we would again prefer to replace  $\Gr\;(r,E_0)$ with (an affine bundle over) $\Gr\;(r,B)$.
For this we require \eqref{Esurj2} to hold, giving
\beq{Lift}
\xymatrix@R=18pt{&& B^* \ar@{->>}[d]\ar@{-->}[dl]_\psi \\
E_1^* \ar[r]^{\s^*}& E_0^* \ar[r]& \coker(\s^*) \ar[r]& 0.}
\eeq
Using the Jouanolou trick described in Section \ref{Jtrick} we may, on replacing $X$ by an affine bundle over it (and pulling everything back to it), assume that $B^*$ is a projective $\cO$-module. Then we may pick a lift $\psi\colon B^*\to E_0^*$.

Over $\Gr\;(r,B)$, with its universal quotient bundle $\pi\colon B\twoheadrightarrow\cQ_B$ (suppressing some pullback maps as usual) we consider the composition
\beq{tor}
\xymatrix@=30pt{
E_0 \ar@/^{-2ex}/[rr]_{\tau}\ar[r]^-{(\sigma,\,\psi^*)}& E_1\oplus B \ar[r]^-{(\id,\,\pi)}& E_1\oplus\cQ_B.}
\eeq
At a closed point $(x,V)\in\Gr\;(r,B)$ (i.e. a point $x\in X$ and an $r$-dimensional subspace $V\le B_x$) we have
$$
\ker\tau_x\=\ker\s_x\cap\ker\;(\pi_x\circ \psi_x^*)\=\ker\s_x\cap V,
$$
where in the last expression we have used $\psi_x^*$ to identify $\ker\s_x$ as a subspace of $B_x$ by \eqref{inx}. In particular $\ker\tau_x$ is at most $r$-dimensional --- so $D_r(\tau)$ is a \emph{deepest} degeneracy locus --- and the locus of points where it is precisely $r$-dimensional is
\beq{taus}
D_r(\tau)\=\big\{(x,U)\in\Gr\;(r,B)\ \colon\ U\subseteq\ker\;(\s_x)\big\}\=\wt D_r(\s).
\eeq
This bijection of sets generalises to $S$-points of $X$: given $f\colon S\to X$ we get a bijection between the sets
$$
\hspace{5mm}\big\{\mathrm{rank\ }r\mathrm{\ subbundles\ }U\subset f^*E_0\ \colon\ U\!\Into\!f^*E_0\rt{\!f^*\s\!}f^*E_1\mathrm{\ is\ zero}\big\}
\vspace{-3mm}$$
\begin{multline*}
\hspace{-3mm}\mathrm{and}\quad\big\{\mathrm{rank\ }r\mathrm{\ subbundles\ }(U\subset f^*E_0,\ V\subset f^*B)\ \colon\text{ the composition} \\
U\!\Into\!f^*E_0\rt{\!f^*\tau\!}f^*E_1\oplus f^*B/V\mathrm{\ is\ zero}\big\}.
\end{multline*}
The functors taking $f\colon S\to X$ to either of these two sets are therefore isomorphic. The first is \eqref{funct} and is represented by $\wt D_r(\s)$ by Proposition \ref{pcut}. The second is represented by $D_r(\tau)\subset\Gr\;(r,B)/X$. Therefore \eqref{taus} is an isomorphism of schemes.

Furthermore we can compare the perfect obstruction theory of Proposition \ref{vcb} for $\wt D_r(\s)$ to the deepest degeneracy locus perfect obstruction theory of \cite[Theorem 3.6]{GT1} for $D_r(\tau)$. In both cases the K-theory class of the virtual cotangent bundle is the restriction of
$$
\Omega_X+\cU\otimes(E_\bu^\vee-\cU^*),
$$
so their virtual cycles agree. Finally, the Thom-Porteous formula for $\wt D_r(\s)$ of Theorem \ref{bar} and the Thom-Porteous formula for $D_r(\tau)$ of \cite[Theorem 3.6]{GT1} both give
$$
\Delta^{r}_{b+e_1-e_0}\big(c\big(\cQ_B-E_\bu\big)\big)
$$
as the pushforward of the virtual cycle to $\Gr\;(r,B)$.

\begin{thm} \label{deep}
Suppose \eqref{Esurj2} holds. Then after pulling back to an affine bundle, the embedding $\iota\_{B}\colon\wt D_r(\s) \into \Gr\;(r,B)$ of \eqref{iot2} is the deepest degeneracy locus of the map of vector bundles \eqref{tor},
$$
\wt D_r(\s)\ \cong\ D_r(\tau).
$$
The two resulting virtual cycles and Thom-Porteous formulae agree.\hfill \qed
\end{thm}

\subsection*{Generalised Carlsson-Okounkov vanishing}
\begin{prop}
Under the assumptions of Theorem \ref{deep},
$$
c_{\;b+e_1-e_0+i}\big(B+E_1-E_0\big)\=0\quad\forall\,i>0.
$$
\end{prop}

\begin{proof}
On an affine bundle over $X$ recall the map \eqref{tor}
\beq{CObdl}
E_0\rt{(\s,\psi^*)}E_1\oplus B
\eeq
of the last Section. By \eqref{Esurj2} it is an injective map of bundles and so is quasi-isomorphic to a rank $b-e_0+e_1$ vector bundle (its cokernel). Therefore its higher Chern classes vanish on the affine bundle, and so on $X$ by \eqref{J=}.
\end{proof}

Taking $B=\cO_X$ and $E_\bu=R\hom_\pi(\cI_1,\cI_2)$ (in the notation of Section \ref{nHs}) over the product of  Hilbert schemes of points $X=S^{[n_1]}\times S^{[n_2]}$ this recovers the original Carlsson-Okounkov vanishing of \cite{CO} and \cite[Corollary 8.1]{GT1}, at least when $H^{\ge1}(\cO_S)=0$ to ensure that $E_\bu$ is 2-term. The general surface $S$ can be handled by another splitting trick to remove $H^{\ge1}(\cO_S)$ from $E_\bu$; see \cite[Section 8]{GT1}.

\section{Comparison}

Given a map of 2-term complexes $F_\bu\to E_\bu$ we can try to compare the virtual cycles of Section \ref{2}. We begin with maps of a specific shape.

\begin{prop} \label{compare}
Suppose given a commutative diagram of vector bundles
\beq{map2term}\xymatrix@R=20pt{
\ F_0\,\ar@{^{(}->}<0pt,-11pt>;[d]<-2pt>^(.35){u_0}\ar[r]^\tau& \ F_1\ \ar@{->>}[d]^{u_1} \\ 
\ E_0\ \,\ar<11pt,-36pt>;[r]^-{\sigma} & \ E_1\ }
\eeq
over a smooth quasi-projective variety $X$, with $u_0$ a vector bundle injection and $u_1$ a surjection. Then there is an embedding
$$
\iota\,\colon\ \wt D_r(F_\bu)\ \Into\ \wt D_r(E_\bu)
$$
such that the tautological bundle $\cU$ on $\wt D_r(E_\bu)$ pulls back to that on $\wt D_r(F_\bu)$ and
$$
\iota_*\big [\wt D_r(F_\bu)\big]^{\vir}\=c_{\;r.\rk(E_\bu-F_\bu)}\big(\cU^*\otimes(E_\bu-F_\bu)\big)\cap \big[\wt D_r(E_\bu)\big]^{\vir}.
$$
\end{prop}

\begin{proof}
Recall that $\wt D_r(E_\bu)$ (and its perfect obstruction theory) is cut out of $\Gr\;(r,E_0)$ by the section $\wt\s$ of $\cU^*\otimes E_1$ \eqref{0loc}. Therefore its virtual cycle is given by the localised top Chern class of $\cU^*\otimes E_1$ \cite[Section 14.1]{Fu},
$$
[\wt D_r(E_\bu)]^{\vir}\=\wt\s\;^!\;[0_{\cU^*\otimes E_1}]\ \in\ A_{\;\dim X+r(e_0-e_1+r)}(\wt D_r(E_\bu)),
$$
where $0_{\cU^*\otimes E_1}\cong\Gr\;(r,E_0)$ is the zero section of $\cU^*\otimes E_1$ and $e_i:=\rk(E_i)$.

Now basechange by the embedding $u\colon\Gr\;(r,F_0)\into\Gr\;(r,E_0)$ induced by the vector bundle injection $u_0\colon F_0\into E_0$. This is a regular embedding cut out by the canonical section of $\cU^*\!\otimes(E_0/F_0)$ on $\Gr\;(r,E_0)$.\footnote{This is a special case of Proposition \ref{pcut} applied to the 2-term complex of bundles $\ker\;(u_0^*)\to E_0^*$ in place of $E_\bu\;$.} Therefore \cite[Proposition 14.1(d)(ii)]{Fu} gives
\beq{que}
u^![\wt D_r(E_\bu)]^{\vir}\=(u^*\tilde\s)^!\big[0_{u^*(\cU^*\otimes E_1)}\big].
\eeq

Set $K_1:=\ker\;(u_1)$. Suppressing pull back maps, on $\Gr\;(r,F_0)/X$ we have the exact sequence
$$
0\To\cU^*\otimes K_1\To\cU^*\otimes F_1\rt{\id\otimes u_1}\cU^*\otimes E_1\To0
$$
in which the section $\wt\tau$ of the second bundle projects to $u^*\wt\s$ in the third. Therefore \eqref{que} and Lemma \ref{fulton} below give the identity
$$
c_{r(f_1-e_1)}(\cU^*\otimes K_1)\cap u^![\wt D_r(E_\bu)]^{\vir}\=j_*\wt\tau^!\big[0_{\cU^*\otimes F_1}\big]\=j_*\big[\wt D_r(F_\bu)\big]^{\vir},
$$
where $j\colon\wt D_r(F_\bu)\into u^{-1}\big(\wt D_r(E_\bu)\big)$ and $f_i:=\rk(F_i)$. Pushing forward by $i\colon u^{-1}\big(\wt D_r(E_\bu)\big)\into\wt D_r(E_\bu)$ and using $\iota=i\circ j$ gives
$$
\iota_*\big[\wt D_r(F_\bu)\big]^{\vir}\=i_*u^!\big(c_{r(f_1-e_1)}(\cU^*\otimes K_1)\cap[\wt D_r(E_\bu)]^{\vir}\big).
$$
We then apply \cite[Example 6.3.4]{Fu} to the fibre square
$$
\xymatrix@R=18pt{\ u^{-1}\big(\wt D_r(E_\bu)\big)\ \ar@{^(->}[r]^-i\ar@{^(->}<0pt,-11pt>;[d]<-2pt>&
\wt D_r(E_\bu)\ar@{^(->}<87pt,-11pt>;[d]<-2pt> \\
\Gr\;(r,F_0)\ \ar@{^(->}[r]^u& \Gr\;(r,E_0)}
$$ 
in which the bottom row is the regular embedding cut out by the canonical section of $\cU^*\!\otimes(E_0/F_0)$. The result is
\beqa
\iota_*\big[\wt D_r(F_\bu)\big]^{\vir} &\!=\!& c_{r(e_0-f_0)}\big(\cU^*\otimes E_0/F_0\big)\cap c_{r(f_1-e_1)}(\cU^*\otimes K_1)\cap[\wt D_r(E_\bu)]^{\vir} \\
&\!=\!& c_{r(e_0-f_0)+r(f_1-e_1)}\big(\cU^*\otimes(E_0/F_0\oplus K_1)\big)\cap[\wt D_r(E_\bu)]^{\vir}.
\eeqa
Substituting $K_1=F_1-E_1$ in K-theory gives the result.
\end{proof}

\begin{lem}\label{fulton}
Let $0\to A\to B\stackrel{\pi\,}\to C\to0$ be an exact sequence of vector bundles over a quasi-projective variety $X$. Given a section $s$ of $B$ we let $t:=\pi(s)$ be the induced section of $C$. Then the resulting localised top Chern classes of $B$ and $C$ are related by the identity
$$
c_a(A)\,\cap\,t^{\;!}\;[0_C]\=\iota_*\;s^!\;[0_B],
$$
where $\iota\colon Z(s)\into Z(t)$ is the embedding of zero schemes and $a:=\rk(A)$.
\end{lem}

\begin{proof}
The total spaces of the bundles fit in the fibre square
$$
\xymatrix@C=30pt{\ A\ \ar@{^(->}[r]^j\ar[d]_\pi& B \ar[dl]+<13pt,10pt><.4ex>\ar[d]^\pi \\
\ X\ \ar@{^(->}[r]<.35ex>^(.6)t\ar@{^(->}[ur]<.35ex>^s& C,\! \ar[l]<.35ex>}
$$
with both maps $\pi$ being \emph{flat}. Letting $0_C$ denote the zero section of $C$, the top Chern class of $C$ localised by $t$ is
$$
t^{\;!}\;[0_C]\=[0_C]\cdot[t(X)]\=[0_C]\cdot\pi_*[s(X)],
$$
where $\,\cdot\,$ denotes the refined intersection product of \cite[Chapter 8]{Fu}, lying in the Chow group of the scheme-theoretic intersection $0_C\cap t(X)=Z(t)$.

So by the projection formula \cite[Example 8.1.7]{Fu},
$$
t^{\;!}\;[0_C]\=\pi_*\big(\pi^*[0_C]\cdot[s(X)]\big)
\=\pi_*\big(j_*[A]\cdot[s(X)]\big).
$$
Now cap with the top Chern class of $A$ to get its zero section:
\beqa
c_a(A)\cap t^{\;!}\;[0_C] &=&
\pi_*\big(j_*[0_A]\cdot[s(X)]\big) \\ \nonumber &=& \pi_*\big([0_B]\cdot[s(X)]\big) \\
&=& \pi_*(s^![0_B]).
\eeqa
But $\pi|_{0_B\cap s(X)}\colon 0_B\cap s(X)\to 0_C\cap t(X)$ is the inclusion $\iota\colon Z(s)\into Z(t)$.
\end{proof}

Proposition \ref{compare} now implies the following comparison results.

\begin{thm}\label{Gee} Let $X$ be a smooth quasi-projective variety. Suppose
\beq{extr}
F_\bu\To E_\bu\To G
\eeq
is an exact triangle in $D(\mathrm{Coh}(X))$ with $F_\bu,\,E_\bu$ being 2-term complexes of vector bundles supported in degrees $0, 1$. If $G$ is quasi-isomorphic to a rank $g$ vector bundle then there is an embedding $\iota\colon\wt D_r(F_\bu)\into\wt D_r(E_\bu)$ such that
\beqa
\iota_*\big[\wt D_r(F_\bu)\big]^{\vir} &=& c_{\;r.\rk\;(E_\bu-F_\bu)}\big(\cU^*\otimes(E_\bu-F_\bu)\big)\cap \big[\wt D_r(E_\bu)\big]^{\vir}\\ &=& c_{rg}(\cU^*\otimes G)\cap \big[\wt D_r(E_\bu)\big]^{\vir}.
\eeqa
Similarly if $G[1]$ is a rank $g$ vector bundle then $\wt D_r(E_\bu)\cong\wt D_r(F_\bu)$ and
\beqa
\big [\wt D_r(E_\bu)\big]^{\vir} &=&
c_{\;r.\rk(F_\bu-E_\bu)}\big(\cU^*\otimes(F_\bu-E_\bu)\big)\cap \big[\wt D_r(F_\bu)\big]^{\vir} \\
&=& c_{rg}(\cU^*\otimes G[1])\cap \big[\wt D_r(F_\bu)\big]^{\vir}.
\eeqa
\end{thm}

\begin{proof}
In the first case fix a representative vector bundle $G$ and a locally free resolution of $E_\bu$, sufficiently negative that the morphism $E_\bu\to G$ in $D(\mathrm{Coh}(X))$ is represented by a genuine map of complexes. Then trim $E_\bu$ by removing $E_{<0}$ and replacing $E_0$ by $E_0/\im(E_{-1})$. Similarly remove $E_{>1}$ and replace $E_1$ by $\ker\;(E_2\to E_3)$. The upshot is a quasi-isomorphic 2-term complex of bundles $\sigma\colon E_0\to E_1$ with a map $f\colon E_0\to G$.

Then since $F_\bu\cong\Cone\;(E_\bu\to G)[-1]$ we can realise the exact triangle \eqref{extr} by the following short exact sequence of vertical complexes
\beq{stan}
\xymatrix@R=0pt{
&&& E_0 \ar@{=}[r]\ar[dd]_(.42){\s\;\oplus\!}^(.42){\!f}& E_0 \ar[dd]^(.42)\s \\ G[-1]\To F_\bu\To E_\bu &= \\
&&G \ar[r]^-{(0,\id)}& E_1\oplus G \ar[r]^(.55){(\id,0)}& E_1}.
\eeq
Therefore the map of 2-term complexes $F_\bu\to E_\bu$ takes the form of \eqref{map2term} with $u_0=\id$ and $u_1=(\id,0)$, and Proposition \ref{compare} now gives the result. \medskip

When $G[1]$ is a vector bundle $H$ a similar argument represents the exact triangle as 
$$
\xymatrix@R=0pt{
&& E_0 \ar@{=}[r]\ar[dd]_{\tau}& E_0 \ar[dd]^\s \\ F_\bu\To E_\bu\To G &= \\
&&F_1 \INTO^-{u_1}& E_1 \ar@{->>}[r]& H.\!}
$$
So now $\wt D_r(F_\bu)$ and $\wt D_r(E_\bu)$ are both cut out of $\Gr\;(r,E_0)$ by sections $\wt\tau$ and $\wt\s=(\id\otimes u_1)\circ\wt\tau$  of $\cU^*\otimes F_1$ and $\cU^*\otimes E_1$ respectively, as in \eqref{0loc}.
Since the former is a subbundle of the latter we find $\wt D_r(F_\bu)\cong\wt D_r(E_\bu)$ and the claimed formula is the excess intersection formula\footnote{Equivalently, the obstruction sheaf of $\wt D_r(E_\bu)$ is an extension of $\cU^*\!\otimes G$ by the obstruction sheaf of $\wt D_r(F_\bu)$, from which the formula also follows.} of \cite[Theorem 6.3]{Fu} applied to the fibre diagram
$$\xymatrix@R=18pt{
\wt D_r \ar[d] \ar[r]& \Gr\;(r,E_0)\ar[d]^-{\wt  \tau}\\ 
\Gr\;(r,E_0) \ar[r]^-{0} \ar@{=}[d] & \cc U^*\otimes F_1\ar[d]_-{\id\!}^-{\!\otimes u_1 }\\
\Gr\;(r,E_0) \ar[r]^-{0} & \cc U^*\otimes E_1}
$$
with excess normal bundle $\cU^*\otimes E_1\big/\cU^*\otimes F_1\cong\cU^*\otimes H$.
\end{proof}

\section{Nested Hilbert schemes}\label{nHs}
A nested Hilbert scheme represents the functor which takes a base scheme $B$ to the set of families of ideals
\beq{jj}
J_1\,\subseteq\,J_2\subseteq\cdots\subseteq\,J_n\,\subseteq\,\cO_{S\times B}
\eeq
with each $\cO_{S\times B}/J_i$ flat over $B$ and of fixed topological class on each $S$ fibre. By \cite[Proof of Lemma 6.13]{Ko} the double duals $J_i^{**}\subseteq\cO_{S\times B}$ are locally free of the form $\cO(-D_i)\subseteq\cO_{S\times B}$ for some Cartier divisor $D_i\subset S\times B$ flat over $B$. Therefore $J_i\subseteq J_i^{**}$ can be written as $I_{Z_i}(-D_i)\subseteq\cO(-D_i)$ for some subscheme $Z_i\subset S\times B$, flat over $B$ of relative dimension 0. Writing $I_i:=I_{Z_i}$ then \eqref{jj} takes the form
\beq{hilbdef}
I_1(-D_1)\,\subseteq\,I_2(-D_2)\subseteq\cdots\subseteq\,I_n(-D_n)\,\subseteq\,\cO_{S\times B}.
\eeq
As we will see in Section \ref{dko} the Hilbert schemes $S_\b$ of pure divisors $\cO(-D)\subseteq\cO_S$ have been heavily studied for their relationship to Seiberg-Witten theory; see \cite{DKO} for instance. Taking the product $S_{\beta_i}\times S^{[n_i]}$ with a (smooth) Hilbert scheme of points gives the Hilbert scheme $S_{\b_i}^{[n_i]}$ of ideals $I_i(-D_i)\subseteq\cO_S$, where $n_i=$\,length\,$(Z_i)$. So to understand \emph{nested} Hilbert schemes what remains is to find a way to impose, one by one, the inclusions
\beq{ii}
I_i(-D_i)\ \subseteq\ I_{i+1}(-D_{i+1})
\eeq
on these products of Hilbert schemes. Taking double duals implies $E_i:=D_i-D_{i+1}$ is effective, so \eqref{ii} is equivalent to the two inclusions
$$
I_i(-E_i)\ \subseteq\ I_{i+1} \quad\mathrm{and}\quad \cO(-D_{i+1})\ \subseteq\ \cO_{S\times B}.
$$
It is sufficient to study the first, since the second is a special case of it. So we are reduced to studying the 2-step nested Hilbert scheme
$$
S_\beta^{[n_1,n_2]}, \qquad\beta\in H^2(S,\Z),\ n_1,n_2\ge0,
$$
which represents the functor mapping $B$ to the set of families of ideals
\beq{in}
I_1(-D)\ \subseteq\ I_2\ \subseteq\ \cO_{S\times B},
\eeq
with $D,\,I_1,\,I_2$ flat over $B$ and such that on closed fibres $S_b$ we have $[D_b]=\beta$ and colength\,$(I_i|_{S_b})=n_i$. The map to double duals recovers $Z_1,\,Z_2$ and $D$, giving a classifying map to a smooth space
\beq{2step}
S_\beta^{[n_1,n_2]}\To S^{[n_1]}\times S^{[n_2]}\times\Pic_\beta(S)\ =:\ X.
\eeq
Since \eqref{in} is the same data as a 1-dimensional subspace of $\Hom(I_1(-D),I_2)$ we see that, at the level of points, $S_\beta^{[n_1,n_2]}$ \eqref{2step} is the virtual resolution $\wt D_1$ of the $r=1$ degeneracy locus of the complex over $X$ which restricts to $R\Hom(I_1,I_2\otimes L)$ at $(I_1,I_2,L)\in X$. To make a scheme theoretic statement we use
$$
\pi\,\colon\ S\times S^{[n_1]}\times S^{[n_2]}\times\Pic_\beta(S)\To S^{[n_1]}\times S^{[n_2]}\times\Pic_\beta(S)
$$
with its universal ideal sheaves $\cI_1,\,\cI_2$ and a fixed choice of Poincar\'e line bundle\footnote{We always normalise $\cL_\b$ (by tensoring it by the pullback of $\cL_\b^{-1}|_{\{x\}\times\Pic_\b(S)}$ if necessary) so that $\cL_\b|_{\{x\}\times\Pic_\b(S)}$ is trivial on some fixed basepoint $x\in S$.} $\cL_\beta$. Using basechange, the Nakayama lemma and the fact that $\Ext^i(I_1,I_2\otimes L)=0$ for $i\not\in[0,2]$ on any $S$ fibre of $\pi$, we can trim
\beq{3term}
R\hom_\pi(\cI_1,\cI_2\otimes\cL_\beta)\ \cong\ \big\{E_0\rt\s E_1\To E_2\big\}\ =:\ E_\bu
\eeq
to be a 3-term complex of locally free sheaves. Using its stupid truncation $E_0\to E_1$ gives the following result, which will be modified later to give versions invariant under quasi-isomorphisms.

\begin{prop} \label{props}
The nested Hilbert scheme \eqref{2step} is the virtual resolution of the $r=1$ degeneracy locus of $\s\colon E_0\to E_1$  \eqref{3term}, giving isomorphisms
$$
S_\beta^{[n_1,n_2]}\ \cong\ \wt D_1(\s)\ \cong\ \PP^*\big(\ext^2_\pi(\cI_2,\cI_1\otimes K_S\otimes\cL_\beta^*)\big)
$$
under which the tautological bundle $\cO(-1)$ on $\wt D_1(\s)$ corresponds to the dual of the tautological quotient bundle on $\PP^*$.
\end{prop}

\begin{proof}
By the $r=1$ case of Proposition \ref{pcut}, $\wt D_1(\s)$ represents the functor which maps an $X$-scheme $f\colon B\to X$ to the set of line subbundles $\curly L\into f^*E_0$ over $B$ factoring through
$$
\curly L\Into h^0(f^*E_\bu)\=\pi\_{B*}\;\hom\big(f^*\cI_1,f^*(\cI_2\otimes\cL_\beta)\big),
$$
where $\pi\_B\colon S\times B\to B$ is the basechange of $\pi$. By adjunction this is equivalent to the set of line bundles $\curly L\to B$ and maps
\beq{ItoI}
f^*\cI_1\Into f^*(\cI_2\otimes\cL_\beta)\otimes\pi_B^*\;\curly L^{-1} \quad\mathrm{over}\ S\times B
\eeq
which are \emph{nonzero on restriction to any fibre $S_b$}. We need to show this is the same as the functor \eqref{in} represented by $S^{[n_1,n_2]}_\beta$, i.e. that \eqref{ItoI} is equivalent to the set of families
\beq{I2I}
f^*\cI_1(-D)\Into f^*\cI_2 \quad\mathrm{over}\ S\times B
\eeq
with $\cI_1,\,\cI_2$ and $D$ flat over $B$ of the correct topological type on each fibre.

Firstly, taking double duals in \eqref{ItoI} gives a map $\cO\to\cO(D)$ for some divisor $D$ which does not contain any fibres and so is flat over $B$. This gives \eqref{I2I} as required. Conversely, given a family \eqref{I2I}, note that $\cO_{S\times B}(D)$ is isomorphic to $f^*\cL_\beta$ on each $S$ fibre, so may be written globally as $f^*\cL_\beta\otimes\pi_B^*\;\curly L^{-1}$, where $\curly L$ is the line bundle $\pi_{B*}\;(f^*\cL_\b(D))$. This gives \eqref{ItoI}, with the map nonzero on every fibre (since $D$ is flat over $B$ and so does not contain any fibre).
\medskip

The second claimed isomorphism follows from the definition \eqref{Grdef} of $\wt D_1(\s)$: it is the space of 1-dimensional quotients of the cokernel of $\s^*\colon E_1^*\to E_0^*$,
$$
\wt D_1(\s)\=\PP^*(\coker\s^*)\=\PP^*\big(h^0(E_\bu^\vee)\big)\=\PP^*\big(\ext^2_\pi(\cI_2,\cI_1\otimes K_S\otimes\cL_\beta^*)\big)
$$
by relative Serre duality down $\pi$.
\end{proof}

\subsection{Embedding and reduced virtual cycle.}\label{redoos} Throughout this Section we require $\beta$ to be of Hodge type (1,1). Until now everything has been true (but vacuous) when this assumption fails, and the results deformation invariant. The constructions of this Section are only invariant under deformations of $S$ inside the Noether-Lefschetz locus in which $\beta$ is of type (1,1).

As discussed in the Introduction, each fibre of $\wt D_1(\s)\to D_1(\s)$ admits a natural embedding
\beq{emb}
\PP\big(\!\Hom(I_1,I_2\otimes L)\big)\ \subseteq\ \PP\big(H^0(L)\big)\,\stackrel{s\_{\!A}}\Into\,\PP\big(H^0(L(A))\big).
\eeq
Here we have fixed, once and for all, $s\_{\!A}\in H^0(\cO_S(A))$ cutting out a divisor $A\subset S$ sufficiently positive that 
$$
H^{\ge 1}(L(A))=0 \quad \quad \forall  L\in \Pic_\b(S).
$$
(When $H^{\ge1}(\cO_S)=0$ and $L$ is effective we may take $A=\emptyset$ and the proof of Theorem \ref{pushBX}  below reduces to a proof of Theorem \ref{ez} instead.) So we let $B$ be the vector bundle
\beq{BB}
B\,:=\,\pi_*(\cL_\beta(A)) \quad\mathrm{over}\ X\,=\,S^{[n_1]}\times S^{[n_2]}\times\Pic_\beta(S).
\eeq
Then \eqref{emb} gives the embedding $h^0(E_\bu)\subset B_x$ of \eqref{Bembed} at each point $x=(I_1,I_2,L)\in X$. The global version \eqref{Esurj2} is the composition of surjections
\begin{align} \nonumber
B^*\,\cong\,R^2\pi_*(\cc L^*_\b(K_S-A))\!\!\onto{\!s_A}\!\!R^2\pi_*(\cL_\beta^*\otimes K_S)\,\cong\,R^2\pi_*(\cI_1\otimes\cL_\beta^*\otimes K_S)
\ &\\ \label{Bsuj}
\onto{}\!\!\ext^2_\pi(\cI_2,\cI_1\otimes\cL_\beta^*\otimes K_S).&
\end{align}
By \eqref{iot2} we get an embedding
\beq{inB}
\iota\_{B}\,\colon\ S_\beta^{[n_1,n_2]}\ \cong\ \wt D_1(E_\bu)\Into\PP(B)\rt{q}X.
\eeq
Suppose now that, as in the statements of Theorems \ref{ez} and \ref{pushBX},
\beq{assu}
H^2(\cO(D))=0 \quad\mathrm{\ for\ any\ effective\ divisor\ }D\mathrm{\ in\ class\ }\beta.
\eeq
The resulting surjection
$$
0\ =\ H^2(L)\ \cong\ \Ext^2(\cO,I_2(L)) \onto{} \Ext^2(I_1,I_2(L))\ \cong \ h^2(E_\bu|_x)
$$
and basechange show that $E_\bu$ has $h^{\ge2}=0$ over the image of $\S{n_1,n_2}_\b$ in $X$, and therefore also over a neighbourhood $U$ thereof. Therefore in \eqref{3term} we can take $E_\bu$ to be a 2-term complex of vector bundles  
\beq{Unhd}
R\hom_\pi(\cI_1,\cI_2\otimes\cL_\beta)\ \cong\ \big\{E_0\rt\s E_1\big\}\,=:\,E_\bu \quad\mathrm{over}\ U\subseteq X.
\eeq
Proposition \ref{props} then gives
$$
\S{n_1,n_2}_\b\ \cong\ \wt D_1(E_\bu)
$$ 
and Proposition \ref{vcb} endows it with a perfect obstruction theory which we call the \emph{reduced} obstruction theory.\footnote{This is because \eqref{VW-h2} shows that at the level of virtual tangent bundles it is obtained from the Vafa-Witten perfect obstruction theory by removing a copy of $H^2(\cO_S)$.}

\begin{thm}\label{robs}
Suppose $H^2(\cO(D))=0$ for every effective divisor $D$ in class $\beta$. The above construction gives a reduced obstruction theory with virtual tangent bundle
\begin{multline} \label{VW-h2}
T^{\vir}_{S^{[n_1,n_2]}_\b}\=
-R\hom_{\pi}(\cI_1,\cI_1)\_0-R\hom_{\pi}(\cI_2,\cI_2)\_0\,+\\[-4pt]
R\hom_\pi\big(\cI_1,\cI_2\otimes\cL_\b(1)\big)+R^1\pi_*\;\cO-\cO 
\end{multline}
in K-theory, and reduced virtual cycle
\beq{redvc}
\big[\S{n_1,n_2}_\b\big]^{\op{red}}\ \in\ A_{n_1+n_2+\vd_\b+p_g}\big(\S{n_1,n_2}_\b\big),
\eeq
where $\cO(1):=\cO_{\bb P(B)}(1)$ and $\vd_\b:=\b(\b-K_S)/2$. \medskip

\noindent If $H^2(L)=0$ for all $L\in\Pic_\b(S)$ then the pushforward of \eqref{redvc} to $\PP(B)$ is
$$
\iota\_{B*}\big[\S{n_1,n_2}_\b\big]^{\op{red}}\=c\_{n_1+n_2+d}\big(B(1)-R\hom_\pi\big(\cI_1,\cI_2\otimes\cc L_\b(1)\big)\big),
$$
where $d=\frac12A.(2\b+A-K_S)=\chi(L(A))-\chi(L)$ for any $L\in\Pic_\beta(S)$.
\end{thm}

\begin{proof}
The virtual tangent bundle is described by the $r=1$ case of  Proposition \ref{vcb} with $\cU=\cO(-1)$, giving the restriction of
$$
T_X-\cO_X-\cO(-1)^*\otimes R\hom_\pi(\cI_1,\cI_2\otimes\cL_\b)
$$
to $\S{n_1,n_2}_\b$. Since $X=S^{[n_1]}\times S^{[n_2]}\times\Pic_\beta(S)$ we derive \eqref{VW-h2} from
$$
T_{\Pic_\beta(S)}\,\cong\,R^1\pi_*\;\cO\quad\mathrm{and}\quad
T_{S^{[n_i]}}\,\cong\,\ext^1_\pi(\cI_i,\cI_i)\_0\,\cong\,R\hom_\pi(\cI_i,\cI_i)\_0[1].
$$

When $H^2(L)=0$ for all $L\in\Pic_\b(S)$ then $E_\bu$ is 2-term over the whole of $X$, so we may apply Theorem \ref{bar} with $r=1$ to give the formula for the pushforward.
\end{proof}

\begin{rmk} This proves Theorems \ref{ez} and \ref{pushBX} from the Introduction. Since the virtual class depends only on the K-theory class of the virtual tangent bundle, Theorem \ref{robs} shows that $\big[\S{n_1,n_2}_\b\big]^{\op{red}}$ coincides with the reduced virtual cycles of \cite{KT1} (when $n_1=0$) and \cite[Theorem 3]{GSY1} (when the latter is defined).
\end{rmk}

\subsection{Comparison of reduced cycles.}\label{comp1}
Next we will compare the reduced cycles on $S^{[n_1,n_2]}_\b$ and $S_\b$. Note that $S^{[n_1,n_2]}_\b$ and $S_\b\times S^{[n_1]}\times S^{[n_2]}$ are the virtual resolutions of the degeneracy loci of
$$
R\hom_\pi(\cI_1,\cI_2\otimes\cL_\b) \quad\mathrm{and}\quad R\pi_*\;\cL_\b\quad\mathrm{over\ }X
$$
respectively. These are related by the obvious diagram
\beq{vert}
\xymatrix@R=18pt{
& R\pi_*\;\cL_\b \ar[d] \\
R\hom_\pi(\cI_1,\cI_2\otimes\cL_\b) \ar@{..>}^u[ur]\ar[r]&
R\hom_\pi(\cI_1,\cL_\b) \ar@{.>}@/^{-2ex}/[u]\ar[d] \\
& H[-1] \ar@{.>}@/^{-2ex}/[u]}
\eeq
with the vertical column an exact triangle.
Here $H:=\ext^2_\pi(\cO/\cI_1,\cL_\b)$, so
$$
H^*\ \cong\ \big(\cc L^*_\b\otimes K_S\big)^{[n_1]}
$$
by relative Serre duality down $\pi$. Thus $H$ is a vector bundle.

So we now use the Jouanolou trick of Section \ref{Jtrick}, pulling back to an affine bundle over $U$, so that $H$ becomes a projective $\cO$-module. In particular its connecting homomorphism to $R\pi_*\;\cL_\b[2]$ in the vertical exact triangle \eqref{vert} is zero over the open neighbourhood $U$ \eqref{Unhd} over which it is supported in degrees $-2$ and $-1$. (Recall our assumption \eqref{assu} that $H^2(\cO_S(D))=0$ for every effective $D$ of class $\beta$.)
So over $U$ we may choose the splittings marked with dotted arrows in \eqref{vert}.

\begin{lem} \label{vbG} $\op{Cone}(u)$ is quasi-isomorphic to a rank $n_1+n_2$ vector bundle.
\end{lem}

\begin{proof} 
Taking $H^0$ of \eqref{vert} at a closed point $x=(I_1,I_2,L)$ of $U$  gives
$$
\xymatrix@=18pt{
& H^0(L) \ar@{=}[d] \\ \Hom(I_1,I_2\otimes L)\ \ar@{^(->}[r]\ar[ur]^{h^0(u)} & \,\Hom(I_1,L).}
$$ 
Since this diagram commutes the diagonal arrow is an injection. Thus, by the Nakayama lemma, $\op{Cone}(u)$, and any basechange of it, has $h^{<0}=0$.

Taking $H^1$ instead gives part of the diagram
$$
\xymatrix@=18pt{
H^1(I_2\otimes L)\ \ar[r] \ar[d]& H^1(L) \ar@{_(->}[d]\ar[r]& 0 \ar[d] \\ \Ext^1(I_1,I_2\otimes L) \ar[r]\ar[ur]^(.35){h^1(u)} & \Ext^1(I_1,L) \ar@/^{-2ex}/[u]\ar[r]^-a& \Ext^1\!\big(I_1,L/(L\otimes I_2)\big).}
$$
Again this commutes and the rows are exact. Since $\Ext^1(I_1,I_2\otimes L)$ surjects onto $\ker\;(a)$, in which $H^1(L)$ lies, the diagonal arrow is surjective. Thus $\Cone\;(u)$ has $h^{\ge1}=0$ and is a vector bundle as claimed.
\end{proof}

So it follows immediately from Theorem \ref{Gee} and the Jouanolou trick that
$$
\big[S^{[n_1,n_2]}_\b\big]^{\op{red}}\=c_{n_1+n_2}\big(\!\Cone\;(u)(1)\big)\cap\big[S_\b\big]^{\op{red}}\times\big[S^{[n_1]}\times S^{[n_2]}\big].
$$
On $S\times S_\beta$ the tautological composition $\cO(-1)\to\pi^*\pi_*\cL_\beta\to\cL_\b$ gives the universal section
\beq{univs}
\cO\rt{s_\b}\cL_\b(1) \quad\mathrm{cutting\ out\ the\ universal\ divisor}\ \cD_\b\subset S\times S_\beta.
\eeq
In particular $\cL_\b(1)\cong\cO(\cD_\b)$ on $S\times S_\beta$.
Thus $\Cone\;(u)(1)$ represents the Carlsson-Okounkov K-theory class of \eqref{COdef}:
$$
\Cone\;(u)(1)\=\mathsf{CO}_\beta^{[n_1,n_2]}\=R\pi_*\;\cO(\cD_\b)-R\hom_\pi(\cI_1,\cI_2(\cD_\b)).
$$
Hence we have proved the second formula in Theorem \ref{comparison}:

\begin{cor} If $H^2(\cO_S(D)) =0$ for all effective divisors in class $\beta$ then
\[
\big[S^{[n_1,n_2]}_\b\big]^{\op{red}}\=
c_{n_1+n_2}\big(\mathsf{CO}_\beta^{[n_1,n_2]}\big) \cap \big[\S{n_1}\times\S{n_2}\big] \times \big[S_\b\big]^{\op{red}}.\vspace{-6mm}
\]
$\hfill\square$
\end{cor}

\subsection{Comparison of virtual cycles.}\label{preaff}
To recover the full perfect obstruction theory of \cite{GSY1} or Vafa-Witten theory on $S_\beta^{[n_1,n_2]}$ from (the virtual resolution of) a degeneracy locus we would have to modify the 3-term complexes $R\hom_\pi(\cI_1,\cI_2\otimes\cL_\b)$ and $R\pi_*\;\cL_\b$ by $H^2(\cO_S)$ terms to make them 2-term.
This cannot be done over $X=S^{[n_1]}\times S^{[n_2]}\times\Pic_\b(S)$.\footnote{A referee suggested the following illustrative example. Let $\beta$ be the fibre class in an elliptic $K3$ surface $S$ and take $n_1=0=n_2$. Then $S_\beta=\PP^1$ and its Vafa-Witten virtual tangent bundle is easily worked out to be $T_{\PP^1}-\cO_{\PP^1}$ (using \eqref{VWTvir}, for example). To see $\PP^1$ as $\wt D_1$ of a complex over $S^{[n_1]}\times S^{[n_2]}\times\Pic_\b(S)=\,$point we must take $\C^2\rt0\C^q$, but this has virtual tangent bundle $T_{\PP^1}-\cO_{\PP^1}^{\oplus q}(1)$.}

However, up on (the product of $S$ with) $S^{[n_1]}\times S^{[n_2]}\times S_\b$ we can use the universal section $s_\b$ \eqref{univs} of $\cL_\b(1)$ --- notice the twist! --- together with a choice of splitting $R\pi_*\;\cO\cong \bigoplus_{i=0}^2R^i\pi_*\;\cO[-i]$, to induce maps
$$
\hspace{5mm}\xymatrix@R=16pt@C=14pt{
& R^2\pi_*\;\cO[-2] \ar[d]^{s\_\b}\ar@{.>}[dl]\ar[dr]\\
R\pi_*\big(\cI_2\otimes\cL_\b(1)\big) \ar[r]&
R\pi_*\;\cL_\b(1) \ar[r]& \pi_*\big((\cO/\cI_2)\otimes\cL_\b(1)\big)&
\hspace{-3mm}\mathrm{on}\ S^{[n_1,n_2]}_\b.}
$$
The lower row is an exact triangle.
We apply the Jouanolou trick of Section \ref{Jtrick}, pulling back from $S^{[n_1]}\times S^{[n_2]}\times S_\b$ to an affine bundle over it. Here $R^2\pi_*\;\cO$ becomes a projective $\cO$-module, so the right hand diagonal arrow becomes zero. Thus we can fill in the dotted arrow, giving a commutative diagram
\beq{ddgg}
\xymatrix@R=16pt@C=-4pt{
& \hspace{-5mm}R^2\pi_*\;\cO[-2]\hspace{-5mm} \ar[dr]\ar[dl]&
H[-2] \ar[d] \\
R\pi_*(\cI_2\otimes\cL_\b(1)) \ar[rr]\ar[d]&&
R\pi_*\;\cL_\b(1) \ar[d] \\
R\hom_\pi(\cI_1,\cI_2\otimes\cL_\b(1)) \ar[rr]&&
R\hom_\pi(\cI_1,\cL_\b(1)).}
\eeq
Here the right hand column is an exact triangle, with $H:=\ext^2_\pi(\cO/\cI_1,\cL_\b(1))$ a vector bundle Serre dual to
$$
H^*\ \cong\ \big(K_S\otimes\cL_\b^{-1}(-1)\big)^{[n_1]}.
$$
On second cohomology sheaves $h^2$ we claim the diagram gives surjections
\beq{sur}
\qquad\xymatrix@R=16pt@C=0pt{
& \hspace{-2mm}R^2\pi_*\;\cO\hspace{-2mm} \ar@{->>}[dl]_a\ar@{->>}[dr]^b\\
R^2\pi_*(\cI_2\otimes\cL_\b(1)) \ar@{=}[rr]^-c\ar@{->>}[d]_d&&
R^2\pi_*\;\cL_\b(1) \ar@{->>}[d]^e \\
\ext^2_\pi(\cI_1,\cI_2\otimes\cL_\b(1)) \ar@{->>}[rr]^-f&&
\ext^2_\pi(\cI_1,\cL_\b(1)).}
\eeq
The claim for the maps $b,c,d,f$ follows from zeros in the obvious exact sequences in which each sits. It then follows from the diagram that $a$ and $e$ are also onto. So if we use the notation
$$
\tau(\cF)\ :=\ \Cone\big(R^2\pi_*\;\cO[-2]\To\cF\big)
$$
for any of the maps $a,b,d\!\;\circ\;\!a$ or $e\!\;\circ\!\;b$ in \eqref{ddgg}, then each $\tau(\cF)$ has $h^{\ge2}=0$ and can be represented by a 2-term complex of vector bundles.
The bottom row and right hand column of \eqref{ddgg} now give the diagram
\beq{vert7}
\xymatrix@R=16pt{
& \tau\big(R\pi_*\;\cL_\b(1)\big) \ar[d] \\
\tau\big(R\hom_\pi(\cI_1,\cI_2\otimes\cL_\b(1))\big) \ar@{.>}[ur]^u\ar[r]&
\tau\big(R\hom_\pi(\cI_1,\cL_\b(1))\big) \ar[d]\ar@{.>}@/^{-2ex}/[u]\\
& H[-1], \ar@{.>}@/^{-2ex}/[u]}
\eeq
which should be compared to \eqref{vert}. Again the column is an exact triangle, and again $H$ is a projective $\cO$-module on the affine bundle so has vanishing connecting homomorphism to the complex $\tau\big(R\pi_*\;\cL_\b(1)\big)[2]$ supported in degrees $-2$ and $-1$. Thus we may split the vertical exact triangle and fill in the dotted arrows on $S^{[n_1]}\times S^{[n_2]}\times S_\b$. \medskip

\noindent\textbf{Extension to ambient space.} Let $B=\pi_*(\cL(A))$ be the bundle \eqref{BB}. Applying \eqref{inB} when $n_1=0=n_2$ gives an embedding $S_\b\subseteq\PP(B)/\Pic_\b(S)$ and so
$$
S^{[n_1]}\times S^{[n_2]}\times S_\b\ \subseteq\ \PP(B)/X,
$$
where $X=S^{[n_1]}\times S^{[n_2]}\times\Pic_\beta(S)$ as in \eqref{2step}.

We apply the Jouanolou trick to $\PP(B)$, giving an affine bundle over it whose total space is an affine variety. The restriction of this affine bundle to $S^{[n_1]}\times S^{[n_2]}\times S_\b$ is therefore also an affine variety, to which the constructions of the last Section \ref{preaff} apply. As usual we suppress pullbacks to these affine bundles.

All of the terms in \eqref{ddgg} extend to (the affine bundle over) $\PP(B)$, as do the horizontal and vertical arrows (canonically). As for the diagonal arrows, their targets are quasi-isomorphic to complexes of vector bundles $G_0\to G_1\to G_2$, and $R^2\pi_*\;\cO$ is a projective $\cO$-module on an affine variety, so the arrows can be represented by maps $R^2\pi_*\;\cO\to G_2$. Again since $R^2\pi_*\;\cO$ is a projective $\cO$-module these maps can be extended to (the affine bundle over) $\PP(B)\supset S^{[n_1]}\times S^{[n_2]}\times S_\b$. So we get an extension of the diagram \eqref{vert7} to (the affine bundle over) $\PP(B)$.

Since the surjectivity \eqref{sur} is an open condition, all the terms of this extended diagram can be represented by 2-term complexes of vector bundles over a neighbourhood $U$ of (the affine bundle over) $S^{[n_1]}\times S^{[n_2]}\times S_\b$ inside (the affine bundle over) $\PP(B)$. By shrinking $U$ if necessary, we may assume that it is affine.\footnote{Let $U=P\take Z$ be an open neighbourhood of a closed subvariety $S$ of an affine variety $P$. Then $Z$ is disjoint from $S$ so we get a surjection $\cO[P]\to\hspace{-4mm}\to\cO[S\sqcup Z]\cong\cO[S]\oplus\cO[Z]$. Let $f$ map to $(1,0)$. Then $P\take f^{-1}(0)\subseteq U$ is an affine neighbourhood of $S$.}

So finally the splitting map $H\to\tau\big(R\hom_\pi(\cI_1,\cL_\b(1))\big)$ of the vertical exact sequence in \eqref{vert7} extends from (the affine bundle over) $S^{[n_1]}\times S^{[n_2]}\times S_\b$ to its affine neighbourhood $U$ since $H|_U$ is a projective $\cO$-module (and hence has vanishing connecting homomorphism to the complex $\tau\big(R\pi_*\;\cL_\b(1)\big)[2]$ supported in degrees $-2$ and $-1$). So we can fill in the dotted arrows, giving 
$$
\hspace{-15mm}\xymatrix@=15pt{
& \tau\big(R\pi_*\;\cL_\b(1)\big) \ar[d] \\
\tau\big(R\hom_\pi(\cI_1,\cI_2\otimes\cL_\b(1))\big) \ar[ur]^u\ar[r]&
\tau\big(R\hom_\pi(\cI_1,\cL_\b(1))\big) \ar@/^{-2ex}/[u]}\vspace{-17mm}
$$
\beq{vert9}
\hspace{9.5cm}\mathrm{over}\ U.\vspace{9mm}
\eeq
Lemma \ref{vbG} now applies verbatim to the map $u$ of \eqref{vert7} to show that
\beq{vbn1}
\Cone\;(u)\text{ is a vector bundle of rank }n_1+n_2\text{ over }U.\footnote{Hence the K-theory class $\mathsf{CO}^{[n_1,n_2]}_\b$ \eqref{COdef} is represented by a bundle after pullback to an affine bundle. By \eqref{J=} this gives a more conceptual proof of the vanishing $c_{n_1+n_2+i}(\mathsf{CO}^{[n_1,n_2]}_\b)=0\ \forall\,i>0$ of \cite[Corollary 8.11]{GT1}. As Aravind Asok kindly explained to us, it is possible that any bundle on the affine bundle is automatically a pullback from the base; this would prove $\mathsf{CO}^{[n_1,n_2]}_\b$ is represented by a bundle on $\S{n_1}\times \S{n_2}\times S_\b$.}
\eeq
Next we describe a nice form of these complexes $\tau(\cF)$ on the open set $U$.

\begin{lem}\label{real}
Over $U$ each of the terms $\tau(\cF)$ in \eqref{vert9} has a 2-term resolution by vector bundles
\beq{rep}
\tau(\cF)(-1)\ \simeq\,\big\{B\To F\big\}
\eeq
which, on dualising, induces the surjection $B^*\to\hspace{-3mm}\to h^0\big((\tau(\cF)(-1))^\vee\big)$ of \eqref{Bsuj}. Moreover the arrows in \eqref{vert9} may be taken to be genuine maps of complexes between these resolutions.
\end{lem}

\begin{proof}
The final claim follows from the fact that on any affine variety like $U$, maps in the derived category between complexes of locally free sheaves are quasi-isomorphic to genuine maps of complexes.

We start by working downstairs on $X$. The composition
\beq{com}
R\pi_*\;\cL_\b\rt{s_A} R\pi_*\;\cL_\b(A)\,=\,\pi_*\;\cL_\b(A)\,=\,B
\eeq
induces the inclusion $H^0(L)\into H^0(L(A))$ on restricting to 
any point $x=(I_1,I_2,L)\in X$ and taking $h^0$. Therefore its cone has cohomology in degrees $1,2$ after basechange to any point, so it can be represented by a 2-term complex of vector bundles $G_1\to G_2$, meaning we can write
$$
R\pi_*\;\cL_\b\ \cong\ \big\{B\To G_1\To G_2\big\}.
$$
By construction this resolution dualises to induce the surjection $B^*\to\hspace{-3mm}\to h^0((R\pi_*\;\cL_\b)^\vee)$ of \eqref{Bsuj}.

Now pullback to (the affine bundle over) $\PP(B)/X$, twist by $\cO_{\PP(B)}(1)$ and use the map $R^2\pi_*\;\cO[-2]\to R\pi_*\;\cL_\b(1)$ of \eqref{ddgg} (which we already noted extends to the affine bundle over $\PP(B)$). It can be represented by a genuine map of complexes, so we can represent its cone as
$$
\tau\big(R\pi_*\;\cL_\b(1)\big)\,\cong\,\big\{B(1)\to G(1)\big\},\ \quad
G:=\ker\big[G_1\oplus(R^2\pi_*\;\cO)(-1)\to\hspace{-3mm}\to G_2\big],
$$
over the affine open $U$ on which the second arrow is a surjection. This gives the result for $\cF=R\pi_*\;\cL_\b(1)$.

For $\cF=\tau\big(R\hom_\pi(\cI_1,\cI_2\otimes\cL_\b(1))\big)$ we can now write the exact triangle defined by $u$ \eqref{vert9} as
$$
\tau\big(R\hom_\pi(\cI_1,\cI_2\otimes\cL_\b(1))\big)\rt{u}\big\{B(1)\to
G(1)\big\}\To\Cone(u)
$$
with the last arrow a genuine map of complexes of vector bundles on  $U$. Thus the first term can be written as the 2-term complex of vector bundles
$$
\tau\big(R\hom_\pi(\cI_1,\cI_2\otimes\cL_\b(1))\big)\ \cong\ \big\{
B(1)\To\Cone(u)\oplus G(1)\big\}.
$$

A similar argument applies to $\tau\big(R\hom_\pi(\cI_1,\cL_\b(1))\big)$.
\end{proof}

So we can now see (the affine bundle over) $S^{[n_1,n_2]}_\b$ as the zeros in $U$ of the section of $F(1)$ defined by the composition
\beq{blar}
\cO_{\PP(B)}(-1)\Into B\To F
\eeq
where $\{B\to F\}$ is the representative \eqref{rep} of $\tau\big(R\hom_\pi(\cI_1,\cI_2\otimes\cL_\b)\big)$. For the confused reader we summarise the construction. \medskip

\noindent\textbf{Summary.}
By Proposition \ref{props}, $S^{[n_1,n_2]}_\b$ is the virtual resolution $\wt D_1$ of the $r=1$ degeneracy locus of the complex $R\hom_\pi(\cI_1,\cI_2\otimes\cL_\b)$ over $X$. Since this is 3-term $\{B\to G_1\to G_2\}$ in general this does not give it a natural virtual cycle, but it does express it as the zero locus in $q\colon\PP(B)\to X$ of the composition\footnote{Changing $G_1$, for instance to $\ker\;(G_1\to G_2)$ when this is locally free, gives the same zero locus but a different perfect obstruction theory.}
\beq{32term}
\cO_{\PP(B)}(-1)\Into q^*B\To q^*G_1.
\eeq
We cannot modify $R\hom_\pi(\cI_1,\cI_2\otimes\cL_\b)$ by an $R^2\pi_*\;\cO$ term on $X$ to get a 2-term degeneracy locus construction directly. But we \emph{can} after pulling back to $\PP(B)$ and twisting by $\cO(1)$, replacing \eqref{32term} by
\beq{22term}
\cO_{\PP(B)}(-1)\Into q^*B\To\ker\big[q^*G_1\oplus(R^2\pi_*\;\cO)(-1)\to\hspace{-3mm}\to q^*G_2\big]
\eeq
on the open set $U$ where the last arrow is a surjection. This is \eqref{blar}.

\begin{thm} \label{33}
$S^{[n_1,n_2]}_\b$ is the zero locus of \eqref{blar}. The virtual cycle of the resulting perfect obstruction theory agrees with the non-reduced virtual cycles of \cite{GSY1,TT1} whenever they are defined.\footnote{The Vafa-Witten moduli space of \cite{TT1} depends on a stability condition. But even in the unstable case, where we get a stack with possibly nontrivial stabilisers, we can pass to the $\C^*$-fixed locus and cut down the automorphisms to those which commute with the $\C^*$ action. The nested Hilbert scheme then appears as a connected component, with the $\C^*$-fixed part of the Vafa-Witten obstruction theory \cite{TT1} defining a perfect obstruction theory on it which recovers those of \cite{GSY1} and Theorem \ref{33}.} It satisfies
$$
\iota_*\big[\S{n_1,n_2}_\b\big]^{\vir}\=c_{n_1+n_2}\big(\mathsf{CO}_\beta^{[n_1,n_2]}\big) \cap\big[\S{n_1}\times \S{n_2}\big] \times  \big[S_\b\big]^{\vir}
$$
under push forward by $\iota \colon \S{n_1,n_2}_\b \Into\S{n_1}\times \S{n_2}\times S_\b$.
\end{thm}

\begin{proof}
Since $\tau\big(R\hom_\pi(\cI_1,\cI_2\otimes\cL_\b)\big)$ differs from $R\hom_\pi(\cI_1,\cI_2\otimes\cL_\b)$ only in $h^{\ge1}$ (even after any basechange), $h^0$ is unaffected. In other words, the zero loci of \eqref{blar} and \eqref{32term} are the same --- and the latter is $S^{[n_1,n_2]}_\b$ by Proposition \ref{props}.

The virtual tangent bundle differs from the one in Theorem \ref{robs} by only $R^2\pi_*\;\cO[-1]$, so its K-theory class is
\beq{VWTvir}
-R\hom_{\pi}(\cI_1,\cI_1)\_0-R\hom_{\pi}(\cI_2,\cI_2)\_0\,+
R\hom_\pi\big(\cI_1,\cI_2(\cD_\b)\big)-R\pi_*\;\cO,
\eeq
agreeing with the ones from \cite{GSY1,TT1}. Since virtual classes depend only on the K-theory class of the virtual tangent bundle, this makes the virtual classes of \cite{GSY1,TT1} equal to the localised top Chern class
\beq{virte}
\big[S^{[n_1,n_2]}_\b\big]^{\vir}\=s^!\;[0_{F(1)}]\ \in\ A_{n_1+n_2+\vd(S_\b)}\big(S^{[n_1,n_2]}_\b\big)
\eeq
of the bundle $F(1)\to U$ and its section $s$ defined by \eqref{blar}.

Setting $n_1=0=n_2$ we get
$$
\big[S_\b\big]^{\vir}\=t^!\;[0_{G(1)}]\ \in\ A_{\vd(S_\b)}\big(S_\b\big),
$$
where $B\to G$ is the resolution of $\tau\big(R\pi_*\;\cL_\b\big)$ on the open subset of (the affine bundle over) $\PP(B\to\Pic_\b(S))$ which is the image of $U$ (in the affine bundle over $\PP(B\to X)$), and $t$ is the induced section of $G(1)$. Pulling back to $U$ gives
\beq{virt2}
\big[S_\b\big]^{\vir}\times S^{[n_1]}\times S^{[n_2]}\=t^!\;[0_{G(1)}],
\eeq
where as usual we have supressed the pull back map on $G(1)$ and its section $t$. Now the map $u$ of \eqref{vert9} induces, via Lemma \ref{real}, a map
$$
\wt u\,\colon\,F(1)\To G(1)
$$
such that $\wt u\circ s=t$. By \eqref{vbn1} $\wt u$ is a surjection with kernel the vector bundle $\Cone\;(u)$, so applying Lemma \ref{fulton} and using \eqref{virt2} gives
\begin{align*}
\big[S^{[n_1,n_2]}_\b\big]^{\vir} &\= s^!\;[0_{F(1)}] \\
&\= c_{n_1+n_2}(\Cone\;(u))\cap t^!\;[0_{G(1)}] \\
&\= c_{n_1+n_2}\big(\mathsf{CO}_\beta^{[n_1,n_2]}\big)\cap\big[S^{[n_1]}\times S^{[n_2]}\big]\times\big[S_\b\big]^{\vir}. \qedhere
\end{align*}
\end{proof}

\begin{rmk} We have now proved Theorem \ref{comparison} from the Introduction. When $p_g(S)>0$ and  $H^2(\cO_S(D))=0$ for all effective $D$ in class $\beta$ --- so the reduced class is defined --- the virtual cycle is automatically zero by the identity
\beq{redvir0}
\big[\S{n_1,n_2}_\b\big]^{\vir}\=c\_{p_g}(\pi_*\;\cO)\cap\big[\S{n_1,n_2}_\b\big]^{\op{red}}\=0
\eeq
of Theorem \ref{Gee}.
\end{rmk}

\section{$\ell$-step nested Hilbert schemes.}\label{lstep}
Fix curve classes $\b_1,\dots, \b_{\ell-1}\in H^2(S,\bb Z)$ and integers $n_1,\dots, n_\ell$. As a set the $\ell$-step nested 
Hilbert scheme is
\begin{multline*} \S{n_1,\dots, n_\ell}_{\b_1,\dots, \b_{\ell-1}}:=\big\{ I_i(-D_i)\subseteq I_{i+1}\subseteq\cO_S, \ i=1,\dots, \ell-1 \ \colon\  \\ [D_i]=\beta_i,\ \mathrm{length}\;(\cO_S/I_i)=n_i\big\}.
\end{multline*}
As a scheme it represents the functor described in \eqref{hilbdef} (with $D_\ell=0$). That description immediately implies that it can be written as the fibre product
$$
\S{n_1,\dots, n_\ell}_{\b_1,\dots, \b_{\ell-1}}\=\S{n_1,n_2}_{\b_1}\times_{\S{n_2}}\S{n_2,n_3}_{\b_2}\times_{\S{n_3}}\cdots \times_{\S{n_{\ell-1}}} \S{n_{\ell-1},n_\ell}_{\b_{\ell-1}}.
$$
It follows that $\S{n_1,\dots, n_\ell}_{\b_1,\dots, \b_{\ell-1}}$ is the intersection of the $(\ell-1)$ virtual resolutions $\wt D_1(E^i_\bu)$ of the degeneracy loci $D_1(E^i_\bu)$ of the complexes
\beq{cpxes}
E^i_\bu\,:=\,R\hom_\pi(\cI_i,\cI_{i+1}\otimes\cL_{\b_i}))\ \,\mathrm{over}\,\ X\,:=\,\prod_{j=1}^\ell\S{n_j}\times \prod_{j=1}^{\ell-1} \Pic_{\b_j}(S).\!\!
\eeq
Here $i=1,\ldots,\ell-1$ while $\pi \colon S\times X\to X$ is the projection and $\cL_{\b_i},\,\cI_i$ are (the pullbacks of) Poincar\'e line bundles and the universal ideal sheaves respectively.

Taking $A\subset S$ to be a divisor which is sufficiently positive that
$$
H^{\ge 1}(L_i(A))\,=\,0 \quad \forall \; L_i \in \Pic_{\b_i}(S),\quad i=1,\dots, \ell-1,
$$
we get vector bundles
$$
B_i\,:=\,\pi_*(\cL_{\b_i}(A)),\quad i=1,\dots, \ell-1, \quad\mathrm{over}\ X
$$
and an embedding
\beq{embeddd}
\iota\_{B}\ \colon\ \S{n_1,\dots, n_\ell}_{\b_1,\dots, \b_{\ell-1}}\Into \PP(B_1)\times\_X \cdots\times\_X\PP(B_{\ell-1})
\eeq
Suppose $H^2(\cO(D_i))=0$ for every effective $D_i$ of class $\b_i,\ i=1,\ldots,\ell-1$. Then each of the complexes \eqref{cpxes} can be represented by a 2-term complex of vector bundles $E_0^i\to E_1^i$. So we can mimic the constructions of Section \ref{2} and \ref{nHs}. We replace the projective bundles $\PP(E_0)/X$ and $\PP(B)/X$ by the fibre products $\PP(E_0^1)\times_X\times\cdots\times_X\PP(E_0^{\ell-1})$ and $\PP(B_1)\times_X\times\cdots\times_X\PP(B_{\ell-1})$, and cut out the $\ell$-step nested Hilbert scheme by the obvious section of $V_1\boxplus\cdots\boxplus V_{\ell-1}$, where $V_i=\cO_{\PP(E_0^i)}(-1)^*\otimes p_i^*(E_1^i)$. The upshot is a (multiply reduced) perfect obstruction theory
on $\S{n_1,\dots, n_\ell}_{\b_1,\dots, \b_{\ell-1}}$ with virtual tangent bundle
\begin{multline*}
-\sum_{i=1}^\ell R\hom_{\pi}(\cI_i,\cI_i)\_0+\sum_{i=1}^{\ell-1}R\hom_\pi(\cI_i,\cI_{i+1}(\cD_{\b_i}))\\ +(R^1\pi_*\cO)^{\oplus (\ell-1)}-\cO^{\oplus (\ell-1)}
\end{multline*}
in K-theory, and reduced virtual class
\beq{reddd}
\Big[\S{n_1,\dots, n_\ell}_{\b_1,\dots, \b_{\ell-1}}\Big]^{\op{red}}\ \in\ A_{n_1+n_\ell+\vd_{\b_1}+\cdots+\vd_{\b_{\ell-1}\!}+(\ell-1)p_g(S)}\Big(\S{n_1,\dots, n_\ell}_{\b_1,\dots, \b_{\ell-1}}\Big).
\eeq
\begin{thm} \label{hpush} If $H^2(L_i)=0$ for all $L_i\in\Pic_{\b_i}(S)$ and all $i=1,\ldots,\ell-1$ then the push forward of the reduced cycle \eqref{reddd} by the embedding \eqref{embeddd} is
\begin{multline*} 
\iota_{B*}\big[ \S{n_1,\dots, n_\ell}_{\b_1,\dots, \b_{\ell-1}}\big]^{\op{red}}\= \\
\prod_{i=1}^{\ell-1} c_{b_i+n_i+n_{i+1}-\chi(\cO_S)-\vd_{\b_i}}\big(B_i(1_i)-R\hom_\pi(\cI_i,\cI_{i+1}(\cc L_{\b_i}(1))\big),\end{multline*}
where $b_i:=\rk(B_i)$ and $\cO(1_i)$ denotes the pullback of $\cO_{\PP(B_i)}(1)$.
\end{thm}

For general classes $\beta_i$ we follow the procedure of Section \ref{preaff} to modify each of the complexes \eqref{cpxes} by $R^2\pi_*\cO[-2]$ after pulling back to (the fibre product of the) $\PP(B_i)$ and tensoring by $\cO(1_i)$.
The result is a perfect obstruction theory with virtual tangent bundle
\beq{vtb2}
-\sum_{i=1}^\ell R\hom_{\pi}(\cI_i,\cI_i)\_0+\sum_{i=1}^{\ell-1} R\hom_\pi(\cI_i,\cI_{i+1}(\cD_{\b_i}))+(R\pi_*\cO)^{\oplus (\ell-1)}
\eeq
in K-theory, and virtual cycle
$$
\Big[\S{n_1,\dots, n_\ell}_{\b_1,\dots, \b_{\ell-1}}\Big]^{\vir}\ \in\  A_{n_1+n_\ell+\vd_{\b_1}+\cdots+\vd_{\b_{\ell-1}\!}}\Big(\S{n_1,\dots, n_\ell}_{\b_1,\dots, \b_{\ell-1}}\Big).
$$
Just as in \eqref{redvir0} this vanishes if $p_g(S)>0$ and $H^2(L)=0$ for all $L\in\Pic_{\b_i}(S)$ for some $i=1,\ldots,\ell-1$.

Iterating the construction of Section \ref{preaff} we get
the following comparison result, a direct generalisation of Theorem \ref{comparison}.

\begin{thm} \label{hcompar} The pushforward of $\big[\S{n_1,\dots, n_\ell}_{\b_1,\dots, \b_{\ell-1}}\big]^{\vir}$ by the embedding
$$
\S{n_1,\dots, n_\ell}_{\b_1,\dots, \b_{\ell-1}} \Into \prod_{j=1}^\ell\S{n_j}\times \prod_{j=1}^{\ell-1} S_{\b_j}
$$
is given by the formula
\begin{multline*}
\prod_{i=1}^{\ell-1}c_{n_i+n_{i+1}}\big(R \pi_*\cO(\cc D_{\b_i})-R \hom_\pi(\cI_i,\cI_{i+1}(\cc D_{\b_i}))\big)\\ \cap \big[\S{n_1}\times\cdots \times  \S{n_\ell}\big] \times  \big[S_{\b_1}\big]^{\vir}\times\cdots \times  \big[S_{\b_{\ell-1}}\big]^{\vir}.
\end{multline*}
Moreover, if $H^2(L_i)=0$ for all effective $L_i\in\Pic_{\b_i}(S),\ i=1,\ldots,\ell-1$ then the same formula holds with all virtual classes $[\ \cdot\ ]^{\vir}$ replaced by reduced classes $[\ \cdot\ ]^{\op{red}}$. \hfill \ $\square$
\end{thm}
 
\begin{rmk} Since \eqref{vtb2} matches the virtual tangent bundle in \cite{GSY1}, the virtual class $\big[\S{n_1,\dots, n_\ell}_{\b_1,\dots, \b_{\ell-1}}\big]^{\vir}$ is the same class that arises in higher rank Vafa-Witten theory \cite{TT1} or the reduced local DT theory of $S$ \cite{GSY2}.
\end{rmk}

\section{Computations} \label{dko}
\noindent\textbf{Vafa-Witten invariants.} On a polarised surface $(S,\cO_S(1))$, Vafa-Witten theory is an enumerative sheaf-counting theory \cite{TT1}. The ``instanton contributions" compute virtual Euler numbers of moduli spaces of Gieseker stable sheaves of fixed rank $r$ and determinant $L$ on $S$, as studied by G\"ottsche-Kool \cite{GK}.

This leaves the ``monopole contributions", which are integrals over moduli spaces of chains of sheaves (of total rank $r$) with nonzero maps between them. When the individual sheaves have rank 1, we get nested ideal sheaves (all tensored by a line bundle). And when $r$ is prime and $p_g(S)>0$, a vanishing theorem \cite[Section 5.2]{TVW} implies that \emph{all} nonzero contributions come from nested Hilbert schemes. A similar result applies in the semistable case \cite{TT2} and the refined version of both theories \cite{TVW}.

The simplest case (for simplicity) is the monopole contribution to the Vafa-Witten invariants in rank 2 and fixed determinant when slope stability is the same as slope semistability. Here we get integrals over the virtual cycles of nested Hilbert schemes of pairs $I_1(-D)\subseteq I_2\subseteq\cO_S$ for any $[D]=\b$ satisfying the slope stability\footnote{Gieseker stability is slightly more subtle, with a further inequality on the $n_i=\,$length$(\cO/I_i)$ in the strictly slope semistable case. Laarakker \cite[Proposition 3.5]{La1} has proved Gieseker unstable Higgs pairs have $\iota_*[\S{n_1,n_2}_\beta]^{\vir}= 0$ so we can sum over all $\beta$, independent of stability.\label{unsta}} condition $\deg(K_S(-D))>0$. We then sum over all $\beta,\,n_1,\,n_2$  giving the same total Chern classes, and then multiply by a factor\footnote{Thanks to Ties Laarakker for pointing out the $r^{2h^1(\cO_S)}$ factor in total rank $r$. This counts the $r$-torsion line bundles we can twist by to get another Higgs pair with the same fixed determinant.} of $2^{2h^1(\cO_S)}$.

The integrand is $1/e(N^{\vir})$, where $N^{\vir}$ is the moving part of the Vafa-Witten obstruction theory and $e$ is the $\C^*$-equivariant Euler class. So by \cite[Section 8]{TT1} we want to compute the integral of
{\small $$
\frac{e(R\hom_\pi(\cI_2,\cI_1(-\cD_\b)K_S\;\t))\ e(R\hom_\pi(\cI_1(-\cD_\b),\cI_2K_S^{-1}\t^{-1}))}
{e(R\hom_\pi(\cI_1,\cI_1K_S\t)\_0)\ e(R\hom_\pi(\cI_2,\cI_2K_S\t))\ e(R\hom_\pi(\cI_2,\cI_1(-\cD_\b)K_S^2\t^2))}\hspace{-4mm}
$$}

\noindent over $\big[S_\b^{[n_1,n_2]}\big]^{\vir}$. Write $\cO(\cD_\b)$ as $\cL_\b(1)$ as in \eqref{univs}. 
Let $h:=c_1(\cO(1))$ on $\PP(B)=\PP(\pi_*\;\cL_\b(A))$ and its subschemes $S_\b^{[n_1,n_2]}$. Then each term $e(\ \cdot\ )$ takes the form
$$
e\big((E_0-E_1)(1)\big)\=\frac{c_{r_0}(E_0(1))}{c_{r_1}(E_1(1))}\=\frac{
\sum_{i=0}^{r_0}c_{r_0-i}(E_0)\cup h^i}{\sum_{i=0}^{r_1}c_{r_1-i}(E_1)\cup h^i}
$$
where the $E_i$ have only nonzero weight spaces (so the quotient makes sense in localised equivariant cohomology) and ranks $r_i$. Therefore, expanding in powers of $h^i$ with coefficients pulled back from $S^{[n_1]}\times S^{[n_2]}\times\Pic_\b(S)$, we get an expression
$$
\sum_i\int_{\big[S_\b^{[n_1,n_2]}\big]^{\vir}}\rho^*\alpha_i\cup h^i\=
\sum_i\int_{\rho_*\big(h^i\cap\big[S_\b^{[n_1,n_2]}\big]^{\vir}\big)}\alpha_i,
$$
where $\rho$ is part of the commutative diagram
\beq{extend}
\xymatrix@R=16pt{
\S{n_1,n_2}_\b \ar[r]^-\iota\ar[dr]_(.4)\rho& S^{[n_1]}\times S^{[n_2]}\times S_\b \ar[d]_(.4){\id\times\!}^(.4){\id\times\,\mathsf{AJ}} \\
& S^{[n_1]}\times S^{[n_2]}\times\Pic_\b(S).}
\eeq
Here $\mathsf{AJ}\colon S_\b\to\Pic_\b(S)$ is the Abel-Jacobi map. So we should calculate the pushdown  $\rho_*\big(h^i\cap\big[S_\b^{[n_1,n_2]}\big]^{\vir}\big)$.
This will prove Theorems \ref{satis}, \ref{notsatis} and \ref{dualite'} from the Introduction, and compute Vafa-Witten invariants in terms of integrals over the smooth spaces $X=S^{[n_1]}\times S^{[n_2]}\times\Pic_\b(S)$.

Let
$$
\mathsf{SW}_\b\ :=\ \left\{\!\!\begin{array}{lll} \deg\,[S_\b]^{\vir} && \vd_\b=\frac12\beta.(K_S-\beta)=0, \\ 0 && \mathrm{otherwise,}\end{array}\right.
$$
denote the Seiberg-Witten invariant \cite{DKO, CK} of $S$ in class $\beta\in H^2(S,\Z)$. Things are easiest when $b^+(S)>1$, i.e. $p_g(S)>0$.

\begin{thm}\label{pg>0} Suppose $p_g(S)>0$. Then in $H_*\big(S^{[n_1]}\times S^{[n_2]}\times\Pic_\b(S)\big)$,
\begin{align*}
\rho_*\big [\S{n_1,n_2}_\b\big]^{\vir}&\=\mathsf{SW}_\beta\cdot c_{n_1+n_2}\big(\!-\!R\hom_\pi(\cI_1,\cI_2\otimes L)\big)\times[L], \\
\rho_*\Big(\;\!h^i\cap\big [\S{n_1,n_2}_\b\big]^{\vir}\Big)&\=0 \quad\mathrm{for\ }i>0,
\end{align*}
where $L\in\Pic_\b(S)$ so $[L]$ is the generator of $H_0(\Pic_\b(S),\Z)$.
\end{thm}

\begin{proof}
If $\,\mathsf{AJ}_*[S_\b]^{\vir}=0$ then $\mathsf{SW}_\b=0$, so by our comparison result Theorem \ref{33} (or Theorem \ref{comparison} in the Introduction) both sides of the claimed identity are zero.
So we assume that $\mathsf{AJ}_*[S_\b]^{\vir}\ne0$.

By \cite[Definition 3.19]{DKO} this means that $\beta$ is a \emph{basic class}. And by \cite[Proposition 3.20]{DKO} $S$ is of \emph{simple type}, which implies that $\vd_\b=0$.


Thus $[S_\b]^{\vir}$ is a 0-dimensional class. Firstly, this implies $h^i|_{[S_\b]^{\vir}}=0$ for $i>0$, and so the second claimed result. And secondly, on restriction to
$$
S\times\S{n_1}\times\S{n_2}\times\big[S_\b\big]^{\vir}
$$
the line bundle $\cL_\b(1)=\cO(\cD_\b)$ is topologically equivalent to (the pull back of) $L$ for any $L\in\Pic_\b(S)$. In particular in the formula of Theorem \ref{33},
\beq{form0}
\iota_*\big[\S{n_1,n_2}_\b\big]^{\vir}\=c_{n_1+n_2}\big(\mathsf{CO}_\beta^{[n_1,n_2]}\big) \cap\big[\S{n_1}\times \S{n_2}\big] \times  \big[S_\b\big]^{\vir}
\eeq
we may replace $\mathsf{CO}^{[n_1,n_2]}_\b$ \eqref{COdef} by $R\pi_*\;L-R\hom_\pi(\cI_1,\cI_2\otimes L)$. Applying $(\id\times\id\times\,\mathsf{AJ})_*$ then gives
\[
\rho_{*}\big [\S{n_1,n_2}_\b\big]^{\vir}\=c_{n_1+n_2}\big(\!-R\hom_\pi(\cI_1,\cI_2\otimes L)\big)\times\deg\big [S_\b\big]^{\vir}\,[\;\mathrm{point}\;],
\]
where $[\;$point$\;]=[L]$ is the generator of $H_0(\Pic_\b(S),\Z)$.
\end{proof}

Now we suppose $b^+(S)=1$, i.e. $p_g(S)=0$. Then $S$ may not be simple type, so we have to consider higher dimensional Seiberg-Witten moduli spaces and the higher Seiberg-Witten invariants \cite{DKO}
$$
\mathsf{SW}_\b^{\;j}\ :=\ \mathsf{AJ}_*\big(h^j \; \cap \; [S_\b]^{\vir}\big)\ \in\ H_{2(\vd_\b-j)}(\Pic_\b(S))\=\wwedge^{\!\;2(\vd_\b-j)}H^1(S).
$$

\begin{thm}\label{pgzero} If $p_g(S)=0$ then in $H_{2(n_1+n_2+\vd_\b-i)}(S^{[n_1]}\times S^{[n_2]}\times\Pic_\b(S))$,
\begin{multline*}
\rho_*\big(h^i\cap\big[\S{n_1,n_2}_\b\big]^{\vir}\big)\= \\
\sum_{j=0}^{n_1+n_2}c_{n_1+n_2-j}\big(R\pi_*\;\cL_\b-R\hom_\pi(\cI_1,\cI_2\otimes\cc L_\b)\big)\cup\mathsf{SW}_\b^{\;i+j}.\qedhere
\end{multline*}
\end{thm}

\begin{proof}
We again use our comparison formula \eqref{form0}. By \cite[Proposition 1]{Ma}\footnote{Manivel extends the formula $c_r(E(1))=\sum_{j=0}^rc_{r-j}(E)\cap c_1(\cO(1))^j$ from rank $r$ bundles to rank $r$ perfect complexes. In fact we showed in \eqref{vbn1} that $\mathsf{CO}^{[n_1,n_2]}_\b$ has the K-theory class of a vector bundle on (an affine bundle over) $S^{[n_1]}\times S^{[n_2]}\times S_\b$ anyway.}
\begin{multline*}
c\_{n_1+n_2}\big(\mathsf{CO}^{[n_1,n_2]}_\b\big)\=
c\_{n_1+n_2}\big(R\pi_*\;\cL_\b(1)-R\hom_\pi(\cI_1,\cI_2\otimes\cc L_\b(1))\big)\\
=\,\sum_{j=0}^{n_1+n_2}c_{n_1+n_2-j}\big(R\pi_*\;\cL_\b-R\hom_\pi(\cI_1,\cI_2\otimes\cc L_\b)\big)\cup h^j.
\end{multline*}
So $(\id\times\id\times\,\mathsf{AJ})_*\big(h^i\,\cap$\,\eqref{form0}$\big)=\rho_*\big(h^i\cap\big[\S{n_1,n_2}_\b\big]^{\vir}\big)$ is
\[
\sum_{j=0}^{n_1+n_2}c_{n_1+n_2-j}\big(R\pi_*\;\cL_\b-R\hom_\pi(\cI_1,\cI_2\otimes\cc L_\b)\big)\cup\mathsf{AJ}_*\big(h^{i+j} \; \cap [S_\b]^{\vir}\big).\qedhere
\]
\end{proof}

We can get simpler formulae by splitting into the cases that $H^2(L)=0$ for all $L\in\Pic_\b(S)$ or not. By Serre duality this is the condition that $\beta^\vee:=K_S-\beta$ is not effective or effective, respectively.

\begin{thm} \label{MHa} Suppose $p_g(S)=0$. Fix $L\in\Pic_\b(S)$ and $i\ge0$. Then the following results hold in $H_{2(n_1+n_2-i)}(S^{[n_1]}\times S^{[n_2]}\times\Pic_\b(S))$. \smallskip

$\bullet$ If $\beta^\vee$ is effective,
$$
\rho_{*}\big(h^i \cap\big[\S{n_1,n_2}_{\b}\big]^{\vir}\big)\=\mathsf{SW}_\b\cdot c_{n_1+n_2+i}\big(\!-R\hom_\pi(\cI_1,\cI_2\otimes L)\big)\times[L],
$$
and both sides vanish when $i>0$. (Cf. Theorem \ref{pg>0}.) \smallskip

$\bullet$ If $\beta^\vee$ is not effective,
$$
\rho_{*} \big(h^i \cap\big [\S{n_1,n_2}_\b\big]^{\vir}\big)\=c_{d+i}\big(\!-R\hom_\pi(\cI_1,\cI_2\otimes\cL_\b)\big),
$$
with $d=n_1+n_2+h^1(\cO_S)-\vd_\b$.
\end{thm}

\begin{proof}
By the comparison result \eqref{form0} both sides of the first identity vanish if $\vd_\b<0$ or $\beta$ is not effective. So we may assume both $\beta,\,\beta^\vee$ are effective and $\vd_\b\ge0$. By \cite[Corollary 3.15]{DKO} this implies $\vd_\b=0$ and $h^1(\cO_S)=1$.

So, just as in \eqref{form0}, $\cO(1)$ is trivial on the 0-dimensional $[S_\b]^{\vir}$ and we may replace $\mathsf{CO}^{[n_1,n_2]}_\b$ in \eqref{form0} by $R\pi_*\;L-R\hom_\pi(\cI_1,\cI_2\otimes L)$ for one $L\in\Pic_\b(S)$. So applying $(\id\times\id\times\,\mathsf{AJ})_*$ to \eqref{form0} gives
$$
\rho_*\big[\S{n_1,n_2}_\b\big]^{\vir}\,=\,c_{n_1+n_2}\big(R\pi_*\;L-R\hom_\pi(\cI_1,\cI_2\otimes L)\big)\times\deg\,[S_\b]^{\vir}\times[\;\mathrm{point}\;].
$$
This deals with $i=0$. When $i>0$ the left hand side of the required identity vanishes by $\vd_\b=0$ and \eqref{form0}, while the right hand side vanishes by the generalised Carlsson-Okounkov vanishing of \cite[Theorem 3]{GT1}. \medskip

For the second result we may assume that $\beta$ is type (1,1), otherwise $\Pic_\b(S)$ is empty and both sides of the identity vanish. Since we also have $H^2(L)=0$ for all $L\in\Pic_\beta(S)$ the reduced cycle of Section \ref{redoos} is defined, and --- since $p_g(S)=0$ --- it equals the virtual cycle. Therefore by Theorem \ref{robs} (or Theorem \ref{pushBX} of the Introduction) the pushforward of the virtual cycle to $\PP(B)$ is given by
\begin{multline*}
c\_{n_1+n_2+k}\big(B(1)-R\hom_\pi(\cI_1,\cI_2\otimes\cc L_\b(1))\big)\=\\
\sum_{j=0}^{n_1+n_2+k}c_{n_1+n_2+k-j}\big(B-R\hom_\pi(\cI_1,\cI_2\otimes\cc L_\b)\big)\cup c_1(\cO(1))^j,
\end{multline*}
again by \cite[Proposition 1]{Ma}. Here $k=\rk(B)-\vd_\b-\chi(\cO_S)$.

We now cup with $h^i=c_1(\cO(1))^i$ and push down $q\colon\PP(B)\to S^{[n_1]}\times S^{[n_2]}\times\Pic_\b(S)$ (which restricts to the map $\rho$ of \eqref{extend}). Using the fact that $q_*\;h^{i+j}$ is the Segre class $s_{i+j-\chi(L(A))+1}(B)$ for $i+j\ge\chi(L(A))-1$, and zero otherwise, we find that $\rho_{*} \big(h^i \cap\big [\S{n_1,n_2}_\b\big]^{\vir}\big)$ is
$$
\sum_{j=\chi(L(A))-1-i}^{n_1+n_2+k}c_{n_1+n_2+k-j}\big(B-R\hom_\pi(\cI_1,\cI_2\otimes\cc L_\b)\big)\cup s_{i+j-\chi(L(A))+1}(B)
\vspace{-5mm}$$
\[
\hspace{25mm}\=c_{n_1+n_2+i+k-\rk(B)+1}\big(-R\hom_\pi(\cI_1,\cI_2\otimes\cc L_\b)\big).\qedhere
\]
\end{proof}

Finally we use the standard duality in Seiberg-Witten theory under $\beta\leftrightarrow\beta^\vee$ to give interesting dualities between invariants of nested Hilbert schemes under $\beta\leftrightarrow\beta^\vee$ and $n_1\leftrightarrow n_2$. Define the map
$$
\rho^\vee\,\colon\ \S{n_2,n_1}_{\b^\vee}\To S^{[n_1]}\times S^{[n_2]}\times \Pic_\b(S)
$$
by replacing $\beta\leftrightarrow\beta^\vee,\ n_1\leftrightarrow n_2$ in
$\rho\colon\S{n_1,n_2}_\b\To S^{[n_1]}\times S^{[n_2]}\times \Pic_\b(S)$
and then composing with $L\mapsto K_S\otimes L^{-1}\colon$ $\Pic_{\b^\vee}(S)\to\Pic_\b(S)$.

\begin{thm}[Duality] In $H_{2(n_1+n_2+\vd_\b-i)}(S^{[n_1]}\times S^{[n_2]}\times\Pic_\b(S))$,
\begin{align*}
\bullet& \text{ if }p_g(S)>0,\quad
\rho_*\big(h^i\cap\big[\S{n_1,n_2}_\b\big]^{\vir}\big)\=(-1)^{s+i}\,\rho^\vee_{*}\big(h^i\cap\big[\S{n_2,n_1}_{\b^\vee}\big]^{\vir}\big), \\
\bullet& \text{ if }p_g(S)=0, \quad\,
\rho_{*}\big(h^i\cap\big[\S{n_1,n_2}_{\b}\big]^{\vir}\big)\=(-1)^{s+i}\;\rho^\vee_{*}\big(h^i\cap\big[\S{n_2,n_1}_{\b^\vee}\big]^{\vir}\big)+ \\
&\hspace{7cm} c_{d+i}\big(\!-R\hom_\pi(\cI_1,\cI_2\otimes\cL_{\b})\big),
\end{align*}
where $s=n_1+n_2-\chi(\cO_S)-\vd_\b=d-1$.
\end{thm}

\begin{proof}
By Theorem \ref{pg>0} with $\b\leftrightarrow\b^\vee,\ n_1\leftrightarrow n_2$ and $L\leftrightarrow K_S\otimes L^{-1}$,
$$
\rho^\vee_*\big[\S{n_2,n_1}_\b\big]^{\vir}\=\mathsf{SW}_{\b^\vee}\cdot
c_{n_1+n_2}\big(\!-R\hom_\pi(\cI_2,\cI_1\cdot K_S\cdot L^{-1})\big)\times[\;\mathrm{point}\;],
$$
for a fixed $L\in\Pic_\b(S)$. For $p_g(S)>0$ Seiberg-Witten invariants have the standard duality
$$
\mathsf{SW}_\b\=(-1)^{\chi(\cO_S)}\;\mathsf{SW}_{\b^\vee}.
$$
(For instance, in \cite{DKO} this is part of Conjecture 0.1, and it is then shown this would follow from proving $\deg\,[S_{K_S}]^{\vir}=(-1)^{\chi(\cO_S)}$ on all minimal general type surfaces. This latter identity was proved in \cite{CK}.) And
$$
c_{n_1+n_2}\big(\!-R\hom_\pi(\cI_2,\cI_1\cdot K_S\cdot L^{-1})\big)=
(-1)^{n_1+n_2}c_{n_1+n_2}\big(\!-R\hom_\pi(\cI_1,\cI_2\cdot L)\big)
$$
by Serre duality down the fibres of $\pi$. This gives the first identity for $i=0$. For $i>0$ it is trivial by Theorem \ref{pg>0}.
\medskip

To prove the second identity we distinguish 3 cases. If $\beta^\vee$ is not effective, i.e. $H^2(L)=0$ for all $L\in\Pic_\b(S)$, then the result is the last part of Theorem \ref{MHa}. Similarly if $\beta$ is not effective then we apply the last part of Theorem \ref{MHa} with $\beta\leftrightarrow\beta^\vee,\,n_1\leftrightarrow n_2$. By $K_S\otimes\cL_{\b^\vee}^{-1}=\cL_\b,\ \vd_\b=\vd_{\b^\vee}$ and Serre duality down the fibres of $\pi$,
$$
c_{d+i}\big(\!-R\hom_\pi(\cI_2,\cI_1\otimes\cL_{\b^\vee})\big)\= (-1)^{d+i}c_{d+i}\big(\!-R\hom_\pi(\cI_1,\cI_2\otimes\cL_{\b})\big),
$$
which gives the required result.

Finally we consider $\beta$ and $\beta^\vee$ both effective. If $\vd_\b<0$ then the left hand side of the identity is zero by the comparison result Theorem \ref{comparison}. On the right hand side, we have
\begin{multline*}
c_{n_1+n_2+h^1(\cO_S)-\vd_\b}\big(\!-R\hom_\pi(\cI_1,\cI_2\otimes\cL_{\b})\big)\= \\ \sum_{i\ge0}c_i(-R\pi_*\;\cL_\b)\cdot
c_{n_1+n_2+h^1(\cO_S)-\vd_\b-i}\big(R\pi_*\;\cL_\b-R\hom_\pi(\cI_1,\cI_2\otimes\cL_{\b})\big).
\end{multline*}
The first term of the sum vanishes for $i>\dim\Pic_\b(S)=h^1(\cO_S)$, while the second term vanishes for $i<h^1(\cO_S)-\vd_\beta$ by the generalised Carlsson-Okounkov vanishing of \cite[Theorem 3]{GT1}. When $\vd_\beta<0$ this makes the whole sum vanish.

So we may assume that $\vd_\b\ge0$ and $\beta,\,\beta^\vee$ are both effective. This implies $\vd_\b=0$ and $h^1(\cO_S)=1$ by \cite[Corollary 3.15]{DKO}.

So, just as in \eqref{form0}, $\cO(1)$ is trivial on the 0-dimensional $[S_\b]^{\vir}$ and we may replace $\mathsf{CO}^{[n_1,n_2]}_\b$ in \eqref{form0} by the pull back of $R\pi_*\;\cL_\b-R\hom_\pi(\cI_1,\cI_2\otimes\cL_\b)$.
Thus pushing \eqref{form0} down by $(\id\times\id\times\,\mathsf{AJ}_\b)_*$ gives $\rho_*\big[\S{n_1,n_2}_\b\big]^{\vir}$ as
\beq{cf}
c_{n_1+n_2}\big(R\pi_*\;\cL_\b-R\hom_\pi(\cI_1,\cI_2\otimes\cL_\b)\big)\cap\big[\S{n_1}\times\S{n_2}\big]\times\mathsf{AJ}_{\b*}\;[S_\b]^{\vir}.
\eeq
Swapping $\beta\leftrightarrow\beta^\vee$ and $n_1\leftrightarrow n_2$ shows $\rho^\vee_*\big[\S{n_2,n_1}_{\b^\vee}\big]^{\vir}$ is
$$
c_{n_1+n_2}\big(R\pi_*\;\cL_{\b^\vee}\!-\!R\hom_\pi(\cI_2,\cI_1\otimes\cL_{\b^\vee})\big)\cap\big[\S{n_1}\times\S{n_2}\big]\times\mathsf{AJ}_{\b^\vee*}\;[S_{\b^\vee}]^{\vir}.
$$
By Serre duality down $\pi$ and the identity $c_i(E)=(-1)^ic_i(E^\vee)$, this is
\begin{align} \nonumber
(-1)^{n_1+n_2}c_{n_1+n_2}\big(\!\;R\pi_*\;\cL_\b-R\hom_\pi(\cI_1,\cI_2\otimes\cL_\b)\big)\,\cap& \\ \big[\S{n_1}\times\S{n_2}\big]&\times\mathsf{AJ}_{\b^\vee*}\; [S_{\b^\vee}]^{\vir}. \label{amn}
\end{align}
We have noted that $h^1(\cO_S)=1$, so $\chi(\cO_S)=0$. Thus, identifying $\Pic_\b(S)\cong\Pic_{\b^\vee}(S)$ via $L\mapsto K_S\otimes L^{-1}$, \cite[Theorem 3.16]{DKO} gives
$$
\mathsf{AJ}_{\b*}\;[S_\b]^{\vir}\=\mathsf{AJ}_{\b^\vee*}\;[S_{\b^\vee}]^{\vir}+c_1(-R\pi_*\;\cL_\b).
$$
Since $\Pic_\b(S)$ is 1-dimensional, $c\_{\ge2}(-R\pi_*\;\cL_\b)=0$. And $c_{\;\ge n_1+n_2+1}(R\pi_*\;\cL_\b-R\hom_\pi(\cI_1,\cI_2\otimes\cL_\b))=0$ by the generalised Carlsson-Okounkov vanishing of \cite[Theorem 3]{GT1}. Therefore
\begin{multline*}
c_1(-R\pi_*\;\cL_\b)\;c_{n_1+n_2}\big(\!\;R\pi_*\;\cL_\b-R\hom_\pi(\cI_1,\cI_2\otimes\cL_\b)\big)\=
 \\c_{n_1+n_2+1}\big(\!-R\hom_\pi(\cI_1,\cI_2\otimes\cL_\b)\big).
\end{multline*}
Substituting this into \eqref{amn} and comparing with \eqref{cf} yields
\begin{multline*}
\rho^\vee_*\big[\S{n_2,n_1}_{\b^\vee}\big]^{\vir}\=(-1)^{n_1+n_2}\rho_*\big[\S{n_1,n_2}_\b\big]^{\vir} \\
-(-1)^{n_1+n_2}c_{n_1+n_2+1}\big(\!-R\hom_\pi(\cI_1,\cI_2\otimes\cL_\b)\big)\,\cap \big[\S{n_1}\times\S{n_2}\big]. 
\end{multline*}
Multiplying by $(-1)^{n_1+n_2}$ and noting $d=n_1+n_2+1$ then gives the required result for $i=0$.

For $i>0$ the restriction of $h^i$ to both $[S_\b]^{\vir}$ and $\big[S^{[n_1,n_2]}_\b\big]^{\vir}$ are trivial (the former by $\vd_\b=0$, the latter by the comparison result Theorem \ref{comparison}).
\end{proof}

\begin{thm} If $p_g(S)>0$ the nested Hilbert scheme contributions to the Vafa-Witten invariants are invariants of $S$'s oriented diffeomorphism type.
\end{thm}

\begin{proof}
We sketch the argument for 2-step nested Hilbert schemes; the general case is no more complicated. Handling stability as in Footnote \ref{unsta}, we express the invariants as a sum over $\beta$ with $\deg\beta<\deg K_S$ of terms 
$$
\mathsf{VW}_{2,\b,n}\ =\ \sum_{n=n_1+n_2} \int_{\big[\S{n_1,n_2}_\b\big]^{\vir}}\frac{1}{e(N^{\vir})}\,.
$$
Let $L$ be any line bundle with $c_1(L)=\beta$. Applying Theorem \ref{pg>0} gives
\begin{equation}\label{diesel}
\mathsf{VW}_{2,\b,n}\ =\ \mathsf{SW}_\b \sum_{n=n_1+n_2} \int_{\S{n_1}\times \S{n_2}} A(\cI_1,\cI_2,L),
\end{equation}
where the integrand $A(\cI_1,\cI_2,L)$ is
{\small
$$
\frac{c_n(-R\hom_\pi(\cI_1,\cI_2 L))\ e(R\hom_\pi(\cI_2L,\cI_1K_S\;\t))\ e(R\hom_\pi(\cI_1,\cI_2LK_S^{-1}\t^{-1}))}
{e(R\hom_\pi(\cI_1,\cI_1K_S\t)\_0)\ e(R\hom_\pi(\cI_2,\cI_2K_S\t))\ e(R\hom_\pi(\cI_2,\cI_1L^{-1}K_S^2\t^2))}\,.$$
}

The coefficient $\mathsf{SW}_\b$ is proved in \cite{DKO, CK} to be a Seiberg-Witten invariant, which is an oriented diffeomorphism invariant when $p_g(S)>0$.\footnote{When $p_g(S)=0$ this need not be quite true; only the unordered pair $(\mathsf{SW}_\b,\mathsf{SW}_{\b^\vee})$ is invariant under oriented diffeomorphisms.} 

We would now like to apply the inductive method of \cite{EGL}, which deals with integrals over single Hilbert schemes $S^{[n_1]}$. This is adapted in \cite[Section 5]{GNY}\footnote{We thank a referee for suggesting this method and pointing out the reference \cite{GNY}.} to deal with products of Hilbert schemes $S^{[n_1]}\times S^{[n_2]}$, essentially by applying \cite{EGL} to the disjoint union $T:=S\sqcup S$ and its Hilbert schemes
$$
T^{[n]}\=\bigsqcup_{n=n_1+n_2} \S{n_1}\times \S{n_2}.
$$
Let $\mathfrak I_1$ and $\mathfrak I_2$ denote the ideal sheaves on $S\times T^{[n]}$ whose restrictions to $S\times \S{n_1} \times \S{n_2}$ are $\cI_1$ and $\cI_2$ respectively. Then \eqref{diesel} can be rewritten
$$
\mathsf{VW}_{2,\b,n}\=\mathsf{SW}_\b \int_{T^{[n]}} A(\mathfrak I_1,\mathfrak I_2,L).$$
The induction writes this as an integral over $T^{[n-1]}\times S$, then $T^{[n-2]}\times S^2$ and so on, finally giving an integral over $S^n$. The first step pulls back along the generically finite map $\psi$, divides by $\deg\psi$, then pushes down $\sigma$, in the diagram
\beq{ff}\xymatrix@=18pt{
T^{[n-1,n]} \ar[r]^-\psi\ar[d]_-\sigma& T^{[n]} \\
T^{[n-1]}\times S.
}\eeq
We split $T^{[n-1,n]}$ into two (unions of) connected components
$$
\left(\bigsqcup_{n_1+n_2=n,\ n_1\ge1}S^{[n_1-1,n_1]}\times S^{[n_2]}\right)\ \sqcup\ 
\left(\bigsqcup_{n_1+n_2=n,\ n_2\ge1}S^{[n_1]}\times S^{[n_2-1,n_2]}\right)
$$
and restrict the diagram to each in turn; we describe the first. On each of its components $\sigma$ is projection down the cone $\PP^*(I_Z)\to S^{[n_1-1]}\times S$ --- where $Z=Z_{n_1-1}\subset S^{[n_1-1]}\times S$ is the universal subscheme --- multiplied by the identity on $S^{[n_2]}$. As such it carries a line bundle $\cO(1)$; it is the kernel of $\cO^{[n_1]}\to\hspace{-3mm}\to\cO^{[n_1-1]}$ in the obvious notation.

By \cite[Equation 5.4]{GNY} and \cite[Equation 10]{EGL} the integrand $A(\mathfrak I_1,\mathfrak I_2,L)$ pulls back to a rational function in $\t$ and the Chern classes of
\begin{enumerate}
\item $R\hom_\pi(\mathfrak I_i,\mathfrak I_j\cdot\xi)$ pulled back from $T^{[n-1]}$, where $\xi$ is a tensor product of powers of $L$ and $K_S$,
\item $\mathfrak I_1, \mathfrak I_2$ pulled back from $T^{[n-1]}\times S$,
\item $L$ and $T_S$ pulled back from $S$, and
\item the line bundle $\cO(1)$.
\end{enumerate}
Taking the coefficient of $\t^0$ gives a polynomial in these terms. Since (1)-(3) are pulled back from $T^{[n-1]}\times S$ we can push down by $\sigma_*$ using the projection formula and \cite[Lemma 1.1]{EGL} to replace powers of $c_1(\cO(1))$ by polynomials in the Chern classes of $\mathfrak I_1$ and $\mathfrak I_2$. This completes the first step.

The second (and later) steps are similar, with \eqref{ff} replaced by
$$
\xymatrix@C=30pt@R=18pt{
T^{[n-2,n-1]}\times S \ar[r]^-{\psi\times\id_S}\ar[d]_-{\sigma\times\id_S}& T^{[n-1]}\times S \\
T^{[n-2]}\times S\times S.}
$$
The pull back and push down handles classes pulled back from either factor of $T^{[n-1]}\times S$ just as before. The new issue is to deal with Chern classes of $\mathfrak I_1$ and $\mathfrak I_2$ on $T^{[n-1]}\times S$. Here we use 
\cite[Equation 5.2]{GNY} to express them as pull backs from $T^{[n-2]}\times S\times S$. So the induction continues, just as in \cite[Proposition 3.1]{EGL}.

The final result is an integral over $S^n$ which can be written as a polynomial in the numbers $c_1(S)^2,\,c_2(S),\,\beta^2$ and $c_1(S).\beta$. See also \cite[Lemma 5.5]{GNY} and \cite[Proposition 7.2]{La1}, or \cite[Section 3]{KT2} when $n_1=0$.

Since $e(S),\,p_1(S)$ and the intersection form are oriented diffeomorphism invariants of $S$, so are $c_1(S)^2,\,c_2(S),\,\beta^2$. This leaves $c_1(S).\beta$, but --- as an observant referee pointed out --- we may assume this equals $-\beta^2$; otherwise $\vd_\beta\ne0$ so $\mathsf{SW}_\beta=0$ and the integral \eqref{diesel} vanishes.
\end{proof}

\bibliographystyle{halphanum}
\bibliography{references}

\bigskip \noindent {\tt{amingh@math.umd.edu}} \medskip

\noindent Department of Mathematics \\
\noindent University of Maryland \\
College Park, MD 20742 \\
USA

\bigskip \noindent {\tt{richard.thomas@imperial.ac.uk}} \medskip

\noindent Department of Mathematics \\
\noindent Imperial College London\\
\noindent London SW7 2AZ \\
\noindent UK

\end{document}